\documentclass[reqno]{amsart}

\usepackage{color}

\usepackage{enumitem}
\usepackage{hyperref}

\newtheorem{theorem}{Theorem}
\newtheorem{lemma}[theorem]{Lemma}
\newtheorem{corollary}[theorem]{Corollary}
\newtheorem{definition}[theorem]{Definition}
\newtheorem{proposition}[theorem]{Proposition}
\newtheorem{remark}[theorem]{Remark}
\newtheorem{notation}[theorem]{Notation}
\newtheorem{example}[theorem]{Example}

\numberwithin{theorem}{section}
\numberwithin{equation}{section}

\begin{document}

\title[Nonuniform hyperbolicity and asymptotic behavior]{New notion of nonuniform exponential dichotomy with applications to the theory of pullback and forward attractors}

\author[J. A. Langa]{Jos\'e A. Langa$^1$}
\thanks{$^1$ Departamento de Ecuaciones Diferenciales y An\'alisis Num\'erico, Universidad de Sevilla, Campus Reina Merdeces, 41012, Sevilla, Spain.}
\email{langa@us.es}

\author[R. Obaya]{Rafael Obaya$^2$}
\thanks{$^2$ Departamento de Matem\'atica Aplicada, E. Ingenier\'ias Industriales, Universidad
	de Valladolid, 47011, Valladolid, Spain, and member of IMUVA, Instituto de
	Investigaci\'on en 
	Matem\'aticas, Universidad de Valladolid, Spain.}
\email{rafoba@wmatem.eis.uva.es}

\author[A. N. Oliveira-Sousa]{Alexandre N. Oliveira-Sousa$^{1,3}$}
\thanks{$^3$ Instituto de Ci\^encias Matem\'aticas e de Computa\c c\~ao, Universidade de S\~ao Paulo, Brasil.}
\email{alexandrenosousa@gmail.com}





\begin{abstract}
	In this work we study nonuniform exponential dichotomies and existence of pullback and forward attractors for evolution processes associated to nonautonomous differential equations. We define a new concept of nonuniform exponential dichotomy, for which we 
	provide several examples, study the relation with the standard notion, and establish a robustness under perturbations. 
	We provide a dynamical interpretation of admissibility pairs related with exponential dichotomies
	to obtain existence of pullback and forward attractors.
	We apply these abstract results for ordinary and parabolic differential equations. 
\end{abstract}

\maketitle

\section{Introduction}


\par In the past decades the theory of asymptotic behavior for nonautonomous differential equations has been widely developed, see for instance \cite{Bortolan-Carvalho-Langa,Bortolan-Carvalho-Langa-book,Arrieta-Crvalho-Rodriguez-00,Barreira-Dragicevi-Valls,Barreira-Valls-Sta,Caraballo-Carvalho-Langa-Sousa,Caraballo-Langa-Obaya-Sanz,Carvalho-Langa-Robison-book} and the references therein. 
This theory studies the dynamical systems generated be the solutions of these differential equations. 
Under sensible conditions all the interesting dynamics generated by its solutions are located in a neighborhood of an \textit{attractor}, which is a compact object that attracts bounded subsets of the phase space, see \cite{Bortolan-Carvalho-Langa-book,Carvalho-Langa-Robison-book,Chepyzhov-Vishik}. 
Another important notion in dynamical systems is \textit{hyperbolicity} 
which means that there exist two main directions that dictates the dynamics: one expanding exponentially and another contracting exponentially, 
see \cite{Barreira-Dragicevi-Valls,Barreira-Valls-Robustness-noninvertible,Bortolan-Carvalho-Langa,Caraballo-Kloeden-Real,Caraballo-Langa-Obaya-Sanz,Chow-Leiva-existence-roughness,Chow-Leiva-existence-unbounded,Zhou-Lu-Zhang-1}. 
One of the reason that these concepts are important is that both are robust under perturbation, see \cite{Bortolan-Carvalho-Langa,Bortolan-Carvalho-Langa-book,Carvalho-Langa-Robinson,Carvalho-Langa-2} for attractors and \cite{Barreira-Dragicevi-Valls,Barreira-Valls-Nonuniform,Barreira-Valls-R,Barreira-Valls-Sta,Henry-1} for hyperbolicity.
In this work we study exponential dichotomies and provide applications to the study of attractors for nonautonomous dynamical systems. 
We propose a new concept of nonuniform exponential dichotomy and study the relation between admissible pairs and the existence of pullback and forward attractors for nonautonomous differential equations.

\par For nonautonomous dynamical systems the appropriated notion of hyperbolicity
is given by an \textit{exponential dichotomy}. Let $A(t)$ be a time dependent linear operator (possibly unbounded), so the differential equation
\begin{equation}
\dot{u}=A(t)u
\end{equation}
is associated with a linear evolution process $\mathcal{T}=\{T(t,s): t\geq s\}$. 
We say that $\mathcal{T}$ admits a \textbf{nonuniform exponential dichotomy} if there exists a family of projections $\{\Pi^u(t): t\in \mathbb{R}\}$
such that $T(t,s)\Pi^u(s)=\Pi^u(t)T(t,s)$ (invariance), $T(t,s): R(\Pi^u(s))\to R(\Pi^u(t))$ is an isomorphism, and 
\begin{equation}\label{eq-int-NEDI}
\begin{split}
&\|T(t,s)\Pi^s(s)\|_{\mathcal{L}(X)}\leq M(s) e^{-\omega(t-s)}, 
\ \ t\geq s\\
&\|T(t,s)\Pi^u(s)\|_{\mathcal{L}(X)}\leq M(s) e^{\omega(t-s)}, 
\ \ t< s,
\end{split}
\end{equation}
where $\Pi^s:=Id_X-\Pi^u$ and $M:\mathbb{R}\to [1,+\infty)$ with some growth of order $e^{\upsilon|s|}$, $\upsilon>0$. If the function $M$ is bounded, we say that $\mathcal{T}$ admits an (uniform) \textit{exponential dichotomy}, see 
\cite{Carvalho-Langa-Robison-book,Henry-1} for the uniform case and
\cite{Barreira-Valls-Sta,Caraballo-Carvalho-Langa-Sousa,Zhou-Zhang,Zhou-Lu-Zhang-1} for the nonuniform. 
Many of these works prove that exponential dichotomies are robust under perturbation, fact that is crucial to study attractors under perturbation, see \cite{Bortolan-Carvalho-Langa-book,Bortolan-Carvalho-Langa,Carvalho-Langa-2}. In this work 
we explore another type of application of exponential dichotomies to the theory of attractors inspired by Longo \textit{et al.} \cite[Section 5]{Longo-Novo-Obaya}.

\par First, we study a new type of nonuniform exponential dichotomy. We consider a linear evolution process 
$\{T(t,s): t\geq s\}$ satisfying all 
the conditions to admit a nonuniform exponential dichotomy except that \eqref{eq-int-NEDI} is modified to
\begin{equation}\label{eq-int-NEDII}
\begin{split}
\|T(t,s)\Pi^s(s)\|_{\mathcal{L}(X)}\leq& M({\color{blue}t}) e^{-\omega(t-s)}, 
\ \ t\geq s\\
\|T(t,s)\Pi^u(s)\|_{\mathcal{L}(X)}\leq& M({\color{blue}t}) e^{\omega(t-s)}, 
\ \ t< s.
\end{split}
\end{equation}
This means that the bound $M$ depends on the final time $t$ instead of the initial time $s$. 
When the nonuniform hyperbolicity is expressed by \eqref{eq-int-NEDII} we refer to this notion as \textbf{nonuniform exponential dichotomy of type II} (or simply \textbf{NEDII}), and the standard one, when \eqref{eq-int-NEDI} holds true, as \textbf{nonuniform exponential dichotomy of type I} (\textbf{NEDI}). 
We prove that a NEDII is a different concept of nonuniform hyperbolicity. 
In fact, we provide several examples of evolution processes 
that admits NEDII and does not admit any NEDI. 
We also show that NEDI and NEDII are complementary notions of nonuniform hyperbolicity, since
under certain conditions it is possible to relate them.
For instance, if a linear 
evolution process admits a NEDI, it is expected that the \textit{dual} evolution process admits a NEDII, and vice-versa. 
In Barreira and Valls \cite{Barreira-Valls-Sta}, they use this type of relation to obtain results for NEDI associated with invertible evolution processes. 
Another relation is that NEDI and NEDII are complementary in half lines, $\mathbb{R}^+$ or $\mathbb{R}^-$,
with one being more general than the other depending in which half-line they are defined. 
This simple relation allows us to determine which nonuniform exponential dichotomy is the ``optimal'' one in each half line.
\par The dual correspondence between NEDI and NEDII allow us to establish a robustness result for NEDII. Inspired by Caraballo \textit{et al.} \cite[Theorem 3.11]{Caraballo-Carvalho-Langa-Sousa}, we provide conditions to obtain that NEDII persists under perturbation. This fact guarantees that NEDII is a reasonable notion of nonuniform hyperbolicity.
Furthermore, the robustness result for NEDII can be applied even in situations where we do not know if NEDI is robust under perturbation, see Example \ref{example-NEDII_w2>y2_w1<y1}. 
Therefore, one of our goals is to show that NEDII is a sensible concept and that the study of NEDII leads to a better comprehension of the notion of nonuniform hyperbolicity. 
\par On the other hand, we will use the concept of \textit{admissible pairs} for 
a non-homogeneous differential equation,
\begin{equation}\label{eq-int-non-homogeneous-equation}
\dot{u}=A(t)u+b(t),\ \ t\in \mathbb{R},
\end{equation}
to study existence of attractors for an evolution process $\{S(t,s): t\geq s\}$ associated with
a nonautonomous differential equations in a Banach space $X$
\begin{equation}\label{eq-int-nonautonomous-nonlinear-equation}
\dot{u}=f(t,u), \ \ t\geq s, \ \ u(s)=u_0\in X.
\end{equation}
A pair of Banach spaces $(\mathfrak{Y},\mathfrak{X})$ is said to be \textit{admissible} for
\eqref{eq-int-non-homogeneous-equation} if for each $b\in \mathfrak{Y}$ there is a solution $u\in \mathfrak{X}$, see \cite{Zhou-Lu-Zhang-16,Latushkin-Timothy-Schnaubelt-98,Henry-1,Zhou-Lu-Zhang-16,Zhou-Zhang}.
Admissible pairs are strongly related with exponential dichotomies, very important to characterize existence of exponential dichotomies, and to prove robustness results, see for instance \cite{Barreira-Dragicevi-Valls,Henry-1,Latushkin-Timothy-Schnaubelt-98,Zhou-Lu-Zhang-16,Zhou-Zhang,Zhou-Lu-Zhang-1}. 
Differently of these works, we propose to apply admissible pairs associated with a nonuniform exponential dichotomy to compute the size of attractors.
In Zhou \textit{et al.} \cite{Zhou-Zhang}, they study families of admissible pairs associated with a real parameter and use these admissible pairs to study robustness of exponential dichotomies.
Inspired by
\cite{Zhou-Zhang}, we present admissible pairs that relates the non-homogeneous term $b$ in \eqref{eq-int-non-homogeneous-equation} with the size of the \textit{attractors} of  \eqref{eq-int-nonautonomous-nonlinear-equation}, using a comparison similar to \eqref{eq-int-dissipativeness-a-cte}.


For the evolution process $\mathcal{S}=\{S(t,s):t\geq s\}$ associated to \eqref{eq-int-nonautonomous-nonlinear-equation} it is possible to associate two notions of attraction: \textit{pullback} and 
 \textit{forward}.
A family of compact sets $\widehat{\mathcal{A}}=\{\mathcal{A}(t): t\in \mathbb{R} \}$ is a \textbf{pullback attractor} for $\mathcal{S}$ if $\widehat{\mathcal{A}}$ is invariant, i.e., $S(t,\tau)\mathcal{A}(\tau)=\mathcal{A}(t)$, for  $t\geq s$, $\widehat{\mathcal{A}}$ pullback attracts bounded sets, i.e., for each bounded set $B\subset X$ and $t\in \mathbb{R}$
\begin{equation}
\lim_{s\to-\infty} dist_H(S(t,s)B, \mathcal{A}(t))=0,
\end{equation}
where $dist_H$ is the Hausdorff semi-distance, see \cite{Carvalho-Langa-Robison-book,Caraballo-Lukaszewicz-Real}.
Similarly, there is a notion of forward attraction.  Let 
$\widehat{D}=\{D(t):t\in\mathbb{R}\}$ be a family of nonempty sets and $B\subset X$, then 
$\widehat{D}$ forward attracts $B$ if for each $s\in \mathbb{R}$, there exists 
\begin{equation}
\lim_{t\to+\infty} dist_H(S(t,s)B, D(t))=0.
\end{equation} 
Then a \textbf{forward attractor} $\mathcal{A}_F$ is a compact set such that is forward attracting, see \cite{Chepyzhov-Vishik,Caraballo-Real-04,Cui-Kloeden-18}. 
Moreover, instead of pullback attracting only bounded sets, the pullback attractor can actually attracts elements of a ``larger'' class of subsets, usually called \textit{universe}, se for instance \cite{Caraballo-Kloeden-Real,Caraballo-Lukaszewicz-Real,Carvalho-Langa-Robison-book,Marin-Real}. For us, an \textbf{universe} $\mathcal{M}$ is a class of all nonempty family of subsets of $X$ that is closed by inclusion, i.e., 
if $\widehat{A}$ and $\widehat{B}$ are families of nonempty subsets of $X$ such that $\widehat{A}\subset \widehat{B}$ (inclusion set by set) and $\widehat{B}\in \mathcal{M}$, then 
$\widehat{A}$ also belongs to $\mathcal{M}$.
For instance, an universe can be a class of functions that grows exponentially, see
\cite{Garcia-Luengo-Marin-Real-2013,Garcia-Luengo-Marin-Real-2012,Lukaszewicz-10}.
A typical condition towards existence of attractors is that the vector field is dissipative, which means that the solutions will be ``absorb'' by a bounded set of the phase space in future time. 
For instance, if $X=\mathbb{R}^N$, $a>0$, $b$ is a bounded real function and 
$f$ satisfies 
\begin{equation}\label{eq-int-dissipativeness-a-cte}
2\langle f(t,u), u\rangle_{\mathbb{R}^N} \leq -a|u|_{\mathbb{R}^N}+b(t), \ \ u\in \mathbb{R}^N.
\end{equation} 
Note that \eqref{eq-int-dissipativeness-a-cte} can be seen as a comparison between \eqref{eq-int-nonautonomous-nonlinear-equation} and the scalar ODE $\dot{x}=-ax+b$ and this leads towards the existence of attractors for \eqref{eq-int-nonautonomous-nonlinear-equation}.
This same approach can be applied for several differential equations, for instance:
 Caratheodory ODEs 
\cite{Longo-Novo-Obaya}, Equations with delays \cite{Caraballo-Langa-Robinson-01}, Heat equations 
\cite{Arrieta-Crvalho-Rodriguez-00}, Navier-Stokes equations 
\cite{Caraballo-Real-04,Garcia-Luengo-Marin-Real-2012,Garcia-Luengo-Marin-Real-2013}, and Damped wave equations \cite{Caraballo-Kloeden-Real}.
\par In this work
we provide conditions inspired in \eqref{eq-int-dissipativeness-a-cte} that allow us to obtain existence of attractors for nonautonomous ordinary and partial differential equations.
 Recently, in Longo \textit{et al.} \cite{Longo-Novo-Obaya} the authors provide conditions like \eqref{eq-int-dissipativeness-a-cte} to obtain attractors associated with skew product semi-flows for nonautonomous Caratheodory ODEs. 
This is done by estimating the size of the solutions of\eqref{eq-int-nonautonomous-nonlinear-equation} by the solutions of the ODE $\dot{x}=a(t)x+b(t)$, according to \eqref{eq-int-dissipativeness-a-cte} and assuming that where the linear part $\dot{x}=a(t)x$ admits some exponential decay given by an \textit{uniform hyperbolicity}.
In this paper we consider the case where this ``exponential decay'' can be nonuniform and we compute the size of the solutions of \eqref{eq-int-nonautonomous-nonlinear-equation} for each possible nonautonomous term $b$ in 
\eqref{eq-int-dissipativeness-a-cte}. 
This analysis is inspired by \textit{admissibility pairs} associated with an \textit{nonuniform exponential dichotomy}, see \cite{Zhou-Zhang}.

\par Consequently, we provide new conditions for existence of forward and pullback attractors from admissible pairs and nonuniform exponential dichotomies.
The idea is to compare the vector field \eqref{eq-int-nonautonomous-nonlinear-equation} with 
\eqref{eq-int-non-homogeneous-equation}, and assume that the linear part of \eqref{eq-int-non-homogeneous-equation} admits a nonuniform exponential dichotomy of type II with bound $M(t)=Me^{\delta|t|}$ and exponent $\alpha>0$, and $b$ belongs to a space of continuous functions with growth 
$e^{\lambda\delta|t|}$, $\lambda\in \mathbb{R}$. 
Then we prove that 
for each universe of families that ``grows'' with order $e^{\gamma|t|}$, $\gamma>0$,
there is a pullback attractor $\{\mathcal{A}_\gamma(t): t\in \mathbb{R}\}$ for
\eqref{eq-int-nonautonomous-nonlinear-equation}
 such that 
$\mathcal{A}_\gamma(t)$ is contained in a ball centered at $0$ with radius $R(t)$ of order $e^{(\lambda+1)\delta|t|}$. 
In particular, if $b$ is exponentially small ($\lambda =-1$) $e^{-\delta|t|}$ the corresponding pullback attractor is uniformly bounded, i.e., $\cup_{t\in \mathbb{R}}\mathcal{A}_\gamma(t)$ is bounded. And if 
$b$ is bounded ($\lambda=0$), it is possible see that the size of the pullback attractor grows as 
$e^{\delta|t|}$.
This relation on the admissible pairs $(\lambda \delta, (\lambda+1)\delta)$ is expected, since the same type of relation is noticed in the pairs of admissibility of \cite{Zhou-Zhang} for difference equations.
The same argument is applied to obtain families of compact sets that forward attracts bounded subsets of $X$. 
Then, we apply these results to ODEs (Section \ref{sec-app-odes}) and to parabolic PDEs (Section \ref{sec-application-pdes}).

\par In Section \ref{sec-ned} we define nonuniform exponential dichotomies of type II (NEDII) for evolution processes. 
We prove a simple result that relates NEDII with the standard one (NEDI) in each half-line ($\mathbb{R}^+$ and $\mathbb{R}^-$). In Subsection \ref{subsec-examples_NEDII} we provide several examples of scalar ODEs that admits NEDII, some of them do not admit any NEDI. Later, these examples are going to be used to obtain existence of nonuniform exponential dichotomies for parabolic PDEs, see Subsection \ref{subsec-existence-NEDII}. 
Finally, we study nonuniform exponential dichotomies for
invertible evolution processes in Subsection \ref{subsec-NED-invertile-ep}, and provide a dual correspondence between NEDI and NEDII. As an application of this result, we establish a robustness result for NEDII. 
\par In Section \ref{sec-existence-attractors} we present a preliminary on the the theory of forward and pullback attractors for evolution processes. The results on existence of attractors rely on strong conditions which are suitable for our applications.  
In Subsection \ref{subsec-forward-attraction} we present a result on the existence of forward attractors for evolution processes. In Subsection \ref{subsec-existence-pullback-attractors} we recall some important notions of pullback attraction for universes. We present a theorem on the existence and uniqueness of pullback attractors that not necessarily belongs to universe that it attracts, which differs from the standard theory of pullback attractors \cite{Carvalho-Langa-Robison-book}, see Remark \ref{remark-pullback-attractors} for details. 

\par In Section \ref{sec-app-odes} we study forward and pullback attractors for nonautonomous ODEs. We apply the admissilibity pair analysis for this case in order to obtain
estimates for the attractors.
All the important ideas are presented in Subsection \ref{subssec-comparison-scalar-odes} where we compare an ODE in $\mathbb{R}^N$ with an scalar ODE. 
For each universe we obtain a corresponding pullback attractor and we also provide conditions to set if they are equal or not. Then, in Subsection \ref{subssec-comparison-systems-odes} we compare our equation with a systems of cooperative ODES and apply the same techniques of the previous subsection. 
\par Finally, in Section \ref{sec-application-pdes} we study the results of the previous sections in an infinite dimensional setting. The goal is to show how the techniques for ODEs can be extended to parabolic PDEs with general boundary conditions: Neumann, Dirichlet and Robin. In Subsection \ref{subssec-comparison-scalar-pdes} we prove existence of pullback attractors and a uniformly bounded family of compact sets that forward attracts every bounded subset of the phase space. 
In Subsection \ref{subsec-existence-NEDII} we provide conditions to obtain examples of evolution processes that admits nonuniform exponential dichotomies. We consider a skew product semiflow that admits a continuous separation (\cite{Polacik-Terescak,Mierczynski-Shen}), which  
 allows us to transfer information from a 1-dimensional ODE to an infinite dimensional one. In this way, for each example of Subsection \ref{subsec-examples_NEDII} we provide a parabolic PDE such that the associated evolution process admits the same type of nonuniform exponential dichotomy of its associated scalar ODE. We conclude with a correspondence between NEDI and NEDII using a parabolic PDE and its adjoint problem similar to the invertible case presented in Subsection \ref{subsec-NED-invertile-ep}.

%
%
%


\section{Nonuniform exponential dichotomies}
\label{sec-ned}
\par In this section we study two definitions of nonuniform exponential dichotomies
for a linear evolution process in a Banach space $X$. 
We propose a new type of nonuniform exponential dichotomy, and study the relations between this new concept and the standard one. 
We also provide several examples and study robustness under perturbation.


\subsection{Preliminaries}
\par We first recall the definition of \textit{evolution process} over a metric space $(X,d)$ with parameters in an interval
$\mathbb{J}=\mathbb{R}, \mathbb{J}=\mathbb{R}^+:=\{t\in \mathbb{R}: t\geq 0\}$ or
$\mathbb{J}=\mathbb{R}^-:=\{t\in \mathbb{R}: t\leq 0\}$.
\begin{definition}
	Let $\mathcal{S}=\{S(t,s)\, ; \, t\geq s, \, t,s\in \mathbb{J} \}$ be a family continuous operators in a metric space $(X,d)$.
	We say that $\mathcal{S}$ is an \textbf{evolution process} in $X$ if
	\begin{enumerate}
		\item $S(t,t)=Id_X$, for all $t\in \mathbb{J}$;
		\item $S(t,s)S(s,\tau)=S(t,\tau)$, for $t\geq s\geq \tau$;
		\item $\{(t,s)\in \mathbb{J}^2 ; \, t\geq s\}\times X\ni (t,s,x)\mapsto S(t,s)x\in X$ is continuous.
	\end{enumerate}
	\par If additionally, the operator $S(t,s)$ is invertible for all $t\geq s$, then we say that $\mathcal{S}$ is an \textbf{invertible evolution process}. In this situation we write
	$\mathcal{S}=\{S(t,s): t,s\in \mathbb{J}\}$, where $S(s,t)$ is the  inverse of $S(t,s)$, for $t\geq s$.
\end{definition}
%

\par The following definition is the standard \textit{nonuniform exponential dichotomy} for linear evolution processes in a Banach space $(X,\|\cdot\|_X)$.
\begin{definition}\label{def-nonuniform-exp-dichotomy-1}
	Let $\mathcal{S}=\{S(t,s) \, ; \, t\geq s\}\subset \mathcal{L}(X)$ be a linear evolution process in a Banach space $(X,\|\cdot\|_X)$. 
	We say that $S$ admits 
	\textbf{nonuniform exponential dichotomy of type I on 
		$\mathbb{J}$,
		or simply NEDI}, if there exists
	a family of continuous projections
	$\{\Pi^u(t)\, ; \, t\in \mathbb{J}\}$ such that
	\begin{enumerate}
		\item $\Pi^u(t) S(t,s)= S(t,s) \Pi^u(s)$, for all $t\geq s$;
		\item $S(t,s)|_{R( \Pi^u(s) ) }$ is an isomorphism for all $t\geq s$,
		and the inverse over $R(\Pi^u(t))$ we denote by
		$S(s,t)$; 
		\item there exist
		$M,\alpha,\beta>0$, and 
		$\delta,\nu\geq 0$ such that
		\begin{equation*}
		\|S(t,s)\Pi^s(s)\|_{\mathcal{L}(X)}\leq Me^{\delta|\textcolor{blue}{s}|} e^{-\alpha(t-s)}, 
		\ \ t\geq s,
		\end{equation*}
		where $\Pi^s(s):=(I-\Pi^u(s))$ for all $s\in \mathbb{J}$ and
		\begin{equation*}
		\|S(t,s)\Pi^u(s)\|_{\mathcal{L}(X)}\leq 
		Me^{\nu|\textcolor{blue}{s}|} e^{\beta(t-s)}, 
		\ \ t< s.
		\end{equation*}
	\end{enumerate}
If $\upsilon=\max\{\delta,\nu\}$ and $\omega=\min\{\alpha,\beta\}$, then
$M(t)=Me^{\upsilon|t|}$ and $\omega>0$ are called \textbf{bound} and \textbf{exponent} of the NEDI on $\mathbb{J}$, respectively.
\end{definition}

\par We present another notion of nonuniform exponential dichotomy with a slight modification over Item (3) of Definition \ref{def-nonuniform-exp-dichotomy-1}.
\begin{definition}\label{def-nonuniform-exp-dichotomy-2}
	Let $\mathcal{S}=\{S(t,s) \, ; \, t\geq s\}\subset \mathcal{L}(X)$ be a linear evolution process in a Banach space $(X,\|\cdot\|_X)$. 
	We say that $S$ admits 
	\textbf{nonuniform exponential dichotomy of type II on $\mathbb{J}$ ,
		or simply NEDII}, 
	if there exists 
	a family of continuous projections $\{\Pi^u(t)\, ; \, t\in \mathbb{J}\}$ such that
	\begin{enumerate}
		\item $\Pi^u(t) S(t,s)= S(t,s) \Pi^u(s)$, for all $t\geq s$;
		\item $S(t,s)|_{R( \Pi^u(s) ) }$ is an isomorphism for all $t\geq s$
		and the inverse over $R(\Pi^u(t))$ we denote by
		$S(s,t)$; 
		\item there exist
		$M,\alpha,\beta>0$ and 
		$\nu,\delta \geq 0$ such that
		\begin{equation}\label{eq-def-bound-stable}
		\|S(t,s)\Pi^s(s)\|_{\mathcal{L}(X)}\leq Me^{\delta|\textcolor{blue}{t}|} e^{-\alpha(t-s)}, 
		\ \ t\geq s
		\end{equation}
		where $\Pi^s(s):=(I-\Pi^u(s))$ for all $s\in \mathbb{J}$ and
		\begin{equation}\label{eq-def-bound-unstable}
		\|S(t,s)\Pi^u(s)\|_{\mathcal{L}(X)}\leq Me^{\nu|\textcolor{blue}{t}|} e^{\beta(t-s)}, 
		\ \ t< s.
		\end{equation}
	\end{enumerate}
\end{definition}

\par In this work we will use the following notations.
\begin{notation}
	Let $\mathcal{S}$ be an evolution process
	that
	admits a nonuniform 
	exponential dichotomy of type $i\in \{I,II\}$, families of projections
	$\Pi^s_i$ and $\Pi^u_i$.
	Then we write 
	\begin{enumerate}
		\item the \textbf{stable set at instant $t$}, 
		$X^s_i(t):=\Pi^s_i(t)X$ and the \textbf{unstable set at the instant $t$};
		$X^u_i(t):=\Pi^u_i(t)X$ for all $t\in \mathbb{J}$;
		\item the \textbf{stable family} $X^s_i:=\{X^s_i(t): t\in \mathbb{J}\}$, and 
		the \textbf{unstable family} $X^u_i:=\{X^u_i(t): t\in \mathbb{J}\}$;
		\item $X^s_i(\alpha,\delta)=\{X^j_i(t): t\in \mathbb{J}\}$ to mean that over stable family the bound is 
		given by $M^s(t)=Me^{\delta|t|}$ and the exponent by 
		$\alpha>0$, see \eqref{eq-def-bound-stable};
		\item $X^u_i(\beta,\nu)=\{X^j_i(t): t\in \mathbb{J}\}$ to mean that over the unstable family the bound is 
		given by $M^u(t)=Me^{\nu|t|}$ and the exponent by 
		$\beta>0$, see \eqref{eq-def-bound-unstable}.
	\end{enumerate}
\end{notation}
\par In the case that $\mathbb{J}=\mathbb{R}$, the names ``stable" and ``unstable" in NEDII have the standard sense of exponential dichotomy, only when 
$\alpha>\delta$ and $\beta>\nu$. In fact, at this situation, for every $s\in \mathbb{R}$ fixed, we see that $S(t,s)\Pi^s(s)\to 0$ as $t\to +\infty$ and $S(t,s)\Pi^u(s)\to 0$ as $t\to -\infty$.
However, there are examples of evolution processes that admits NEDII with $\alpha<\delta$ or $\beta<\nu$ with some interesting properties to be explored. For instance, 
 even in this ``pathological situation'', it is possible 
  to obtain applications on the asymptotic behavior for evolution processes (Section \ref{sec-app-odes} and Section \ref{sec-application-pdes}). 
\par The following result provides a simpler way to relate both types of nonuniform exponential dichotomies.
\begin{theorem}\label{th-trivial-relation-between-type1-type2}
	Let $\mathcal{S}$ be an evolution process.
	\par On $\mathbb{R}^+$:
	\begin{enumerate}
		\item there exists a NEDI with $X^s_I(\alpha,\delta)$,
		if and only if,
		there exists a NEDII with $X^s_{II}(\alpha+\delta,\delta)$;
		\item there exists a NEDI with $X^u_I(\beta+\nu,\nu)$, 
		if and only if,
		there exists a NEDII with $X^u_{II}(\beta,\nu)$.
	\end{enumerate}
	\par On $\mathbb{R}^-$:
	\begin{enumerate}
		\item there exists a NEDII with $X^s_{II}(\alpha,\delta)$,
		if and only if,
		there exists a NEDI with $X^s_{I}(\alpha+\delta,\delta)$;
		\item there exists a NEDII with $X^u_{II}(\beta+\nu,\nu)$, 
		if and only if,
		there exists a NEDI with $X^u_{I}(\beta,\nu)$.
	\end{enumerate}
\end{theorem}

\begin{proof}
	Note that, if $t,s\in \mathbb{R}^+$ we have
	\begin{equation*}
	-\alpha(t-s)+\delta s
	=
	-(\alpha+\delta)(t-s)+\delta t.
	\end{equation*}
	Hence, for an evolution process $\mathcal{S}$ that admits NEDI with bound on the stable set 
	$M^s(s)=Me^{\delta|s|}$, for some $M,\delta>0$ (the case $\delta=0$ is trivial), and exponent $\alpha>0$ we have that 
	\begin{equation*}
	\|S(t,s)\Pi^s(s)\|_{\mathcal{L}(X)}\leq Me^{-\alpha(t-s)+\delta| s|}=
	Me^{-(\alpha+\delta)(t-s)+\delta|t|},
	\end{equation*}
	which finishes the proof of Item (1).
	Similarly, Item (2)  follows from the relation
	\begin{equation*}
	\alpha(t-s)+\delta |t|
	=
	(\alpha+\delta)(t-s)+\delta |s|, \ \ t,s\in \mathbb{R}^+.
	\end{equation*}
	\par The proof on $\mathbb{R}^-$ is similar to the case on $\mathbb{R}^+$.
\end{proof}

\par Next corollary summarize the relations of Theorem \ref{th-trivial-relation-between-type1-type2}. 

\begin{corollary}\label{corollary-trivial-information}
	Let $\mathcal{S}$ be an evolution process in a semi-line, i.e., $\mathbb{J}=\mathbb{R}^+$ or $\mathbb{R}^-$.
	If $\mathcal{S}$ admits NEDI (or NEDII) with bound $M(t)=Me^{\upsilon|t|}$ and exponent
		 and $\omega >\upsilon$, then
		$\mathcal{S}$ admits NEDII (NEDI) with bound $M(t)=Me^{\upsilon|t|}$ and exponent $\omega-\upsilon>0$.
\end{corollary}

\par The analysis as in Corollary \ref{corollary-trivial-information} is not optimal, we lose information when unifying the exponents $\alpha, \beta$ and the growth of the bound or order $e^{\delta|t|}$ or $e^{\nu|t|}$.
Note that the same problem occurs when we study the exponents in the whole line. Hence, to provide an ``optimal'' analysis on the relation of the exponents and the growth of the bound, we sometimes consider different exponents, even in the half-lines $\mathbb{R}^+$ and $\mathbb{R}^-$.

%
%


\subsection{Examples of NEDII}\label{subsec-examples_NEDII} In this subsection we provide examples of scalar evolution processes that admit nonuniform exponential dichotomies (of type I and II). Our goal is to guarantee that NEDII is a new concept and explore the differences with the standard notion. Additionally, each example of this subsection can be used to provide an example of parabolic PDEs with nonuniform exponential dichotomies, see Subsection \ref{subsec-existence-NEDII}.
%
\par The following proposition is inspired by Barreira and Valls \cite[Proposition 2.3]{Barreira-Valls-Sta}. 
\begin{proposition}\label{Example-Barreira}
	Let $a,b>0$ and $\mathcal{S}=\{S(t,s): t\geq s\}$ be
	the evolution process defined by 
	$x(t,s;x_0):=S(t,s)x_0$, where $x$ is the solution for $\dot{x}=-b x-at\sin(t)x$, $t\geq s$
	at the initial data $x(s)=x_0\in \mathbb{R}$. We have that 
	\begin{enumerate}
		\item $\mathcal{S}$ admits a NEDII on $\mathbb{R}^+$ with
		$X_{II}^s(b+a,2a)$ and $\Pi^u(t)=0$ for all $t\geq 0$.
		\item $\mathcal{S}$ admits a NEDI on $\mathbb{R}^-$ with
		$X_I^s(b+a,2a)$ and $\Pi^u(t)=0$ for all $t\geq 0$.
			\end{enumerate}
		Additionally, if $b>a$, then 
			\begin{enumerate}
			
		\item $\mathcal{S}$ admits a NEDI on $\mathbb{R}^+$ with
		$X_I^s(b-a,2a)$ and $\Pi^u(t)=0$ for all $t\geq 0$.
		\item $\mathcal{S}$ admits a NEDII on $\mathbb{R}^-$ with
		$X_{II}^s(b-a,2a)$ and $\Pi^u(t)=0$ for all $t\geq 0$.

		\item $\mathcal{S}$ admits NEDI and NEDII on $\mathbb{R}$, with
		$X^s_{j}(b-a,2a)$, $j=I,II$, and $\Pi^u(t)=0$ for all $t\geq 0$.
	\end{enumerate}	
\end{proposition}
\begin{proof}
	Note that $S(t,s)x=e^{-b (t-s)+ at\cos(t)-as\cos(s)-a\sin(t)+a\sin(s)}x$, $t,s\in \mathbb{R}$ and $x\in \mathbb{R}$. 
	Hence
	\begin{equation*}
	\|S(t,s)\|_{\mathcal{L}(\mathbb{R})}=S(t,s) \,1=e^{-(b+a) (t-s)+ at(\cos(t)+1)-as(\cos(s)+1)-a\sin(t)+a\sin(s)}.
	\end{equation*}
	Thus, 
	\begin{eqnarray*}
	S(t,s)\, 1\leq e^{2a-(b+a) (t-s)+2at}, \ t\geq s\geq 0,\\
	S(t,s)\, 1\leq e^{2a-(b+a) (t-s)+2a|s|}, \ s\leq t\leq 0.
	\end{eqnarray*}
	which finishes the prove of the first two items. 
	The proof of the remaining items follows from Theorem \ref{th-trivial-relation-between-type1-type2}.
	\end{proof}
\begin{remark}\label{remark_NEDII>NEDI_semilne}
	Let $\mathcal{S}$ be the evolution process defined in Proposition \ref{Example-Barreira}
	Note that, if $b<3a$,
	$\mathcal{S}$ admits a
	NEDII on $\mathbb{R}^+$ with
	$X_{II}^s(\alpha_2,\delta_2)$, where $\alpha_2:=b+a=>2a=:\delta_2$ and a NEDI on $\mathbb{R}^+$ with $X_I^s(\alpha_1,\delta_1)$
	$\alpha_1:=b-a<2a=:\delta_1$. 
	Hence, in some situations, it is possible to choose NEDII with ``better'' relation in the exponents than NEDI. 
	Of course, an analogous relation its obtained over
	$\mathbb{R}^-$, but $\mathcal{S}$ admits NEDI and NEDII, where NEDI has the ``better'' relation on the exponents for NEDI. 
\end{remark}
\par Next, we provide an example of an evolution process that admits NEDII with two different projections and does not admit any NEDI.
\begin{proposition}\label{prop-example-NEDII-line}
	Define $f_0:\mathbb{R}\to \mathbb{R}$ by 
	\begin{equation}\label{example-f_0-descontinuous}
	f_0(t)= \left\{ 
	\begin{array}{l l} 
	1
	& \quad \hbox{ if } t\geq 0,\\
	-1
	& \quad \hbox{ if } t < 0,
	\end{array} 
	\right.
	\end{equation} 
	and consider the Caratheodory differential equation
	\begin{equation*}
	\dot{x}=f_0(t)x, \hbox{ for } t\in \mathbb{R}.
	\end{equation*}
	Then, the induced evolution process 
	$\mathcal{S}_0=\{S_{0}(t,s):t, s\in \mathbb{R} \}$ admits a NEDII on $\mathbb{R}$ with two different 
	families of projections, and does not admit any NEDI on $\mathbb{R}$.
\end{proposition}

\begin{proof}
	For each $t\in\mathbb{R}$ define the real function $T(t):\mathbb{R}\rightarrow\mathbb{R}$ 
	as
	\begin{equation*}
	T(t)x= \left\{ 
	\begin{array}{l l} 
	e^{t}x 
	& \quad \hbox{if } t\geq 0,
	\\ e^{-t}x \, 
	& \quad \hbox{if } t\leq 0,
	\end{array} 
	\right.
	\end{equation*} 
	for each $x\in \mathbb{R}$. 
	Note that $T(t)$ is an homeomorphism on $\mathbb{R}$ and that
	$S_0(t,s)=T(t)T(s)^{-1}$ for every $t, s\in \mathbb{R}$.
	\par First we show that $\mathcal{S}_0$ admits a NEDII 
	with the $\Pi^u(t)=0$ for all $t\in \mathbb{R}$, i.e., we prove that $\mathcal{S}_0$
	satisfies
	\begin{equation}\label{eq-example-line-NEDII-Not-NEDI}
	\|S_0(t,s)\|_{\mathcal{L}(\mathbb{R})}=S_0(t,s) \,1\leq e^{-(t-s)+2|t|}, \hbox{ for all } t\geq s.
	\end{equation}
	Indeed,
	if 
	$s\leq 0\leq t$
	we are able to write
	$S_0(t,s)1=e^{s+t}= e^{-(t-s)+2(t+s)}$.
	Now, let $t\geq s\geq 0$ then 
	\begin{equation*}
	t-s\leq -(t-s)+2t=-(t-s)+2|t|.
	\end{equation*}
	Thus
	$	S_0(t,s) 1=e^{t-s}\leq e^{-(t-s)+2|t|}$, for $t\geq s\geq 0$.
	Finally, if $s\leq t\leq 0$ then $S_0(t,s) 1=e^{-(t-s)}$ and 
	$\mathcal{S}_0$ satisfies \eqref{eq-example-line-NEDII-Not-NEDI}.
	%
	\par Similarly, it is possible to prove that
	\begin{equation}
	S_0(t,s) \, 1\leq e^{t-s+2|t|}, \  t\leq s.
	\end{equation}
	Therefore, $\mathcal{S}_0=\{S_0(t,s): t\geq s\}$ admits a NEDII with $\Pi^u(\cdot)=0$ and $\Pi^u(\cdot)=Id_\mathbb{R}$.
	\par Finally, suppose that 
	$\mathcal{S}_0$ admits a NEDI on $\mathbb{R}$,
	then exists $\{\widetilde{\Pi}^u(t)\, :\, t\in\mathbb{R} \}$ a family of projections so that
	satisfies all the conditions from the Definition 
	\ref{def-nonuniform-exp-dichotomy-1}.
	It is straightforward to verify that $\widetilde{\Pi}^u(\cdot)$ must be constant equal to the identity map $Id_\mathbb{R}$ or the null operator $0$.
	\par Assume that $\widetilde{\Pi}^u=Id_\mathbb{R}$. 
	Then there are
	 $\widetilde{M},\tilde{\beta}>0$, and $\tilde{\nu}\geq 0$ such that
	\begin{equation}
	S_0(t,s) 1\leq \widetilde{M}e^{\tilde{\beta}(t-s)+ \tilde{\nu}|s| }, \  t\leq s.
	\end{equation}
	Then, for each $s\in \mathbb{R}$ fixed, 
	$S_0(t,s)1\to 0$ as $t\to -\infty$,
	which is a contradiction. 
	\par Similarly, we prove that we can not have $\widetilde{\Pi}^u=0$, and therefore 
	$\mathcal{S}_0$ does not admit any NEDI.
\end{proof}

\begin{proposition}
	\label{Example-f-continuous-bounded}
	Let $f:\mathbb{R}\to \mathbb{R}$ be a continuous function such that
	\begin{enumerate}
		\item $\lim_{t\to +\infty} f(t)=1$;
		\item $\lim_{t\to -\infty} f(t)=-1$.
	\end{enumerate}
	Then, the evolution process $\mathcal{S}_f=\{S_f(t,s): t, s\in \mathbb{R}\}$, associated to $\dot{x}=f(t)x$, admits a NEDII on $\mathbb{R}$ with two different projections. 
	Moreover, $\mathcal{S}_f$ does not admits any NEDI on $\mathbb{R}$.
\end{proposition}

\begin{proof}
	Let $f_0$ be the function defined in \eqref{example-f_0-descontinuous}. 
	Note that
	\begin{equation*}
	\lim_{|t-s|\to +\infty}\frac{1}{t-s}\int_{s}^{t}|f(r)-f_0(r)|dr=0.
	\end{equation*}
	Then, for any $\epsilon\in (0,1)$, there exists $K_\epsilon>0$ such that
	\begin{equation*}
	\int_{s}^{t}|f(r)-f_0(r)|dr\leq  K_\epsilon+\epsilon|t-s|, \hbox{ for each }t, s\in \mathbb{R},
	\end{equation*}
	which yields to
	\begin{equation}\label{eq-relation-integrals-ergodic}
	\int_{s}^{t}f(r)dr\leq \int_{s}^{t} f_0(r)dr+ K_\epsilon+\epsilon|t-s|, \hbox{ for each }t, s\in \mathbb{R}.
	\end{equation}
	Thus
	\begin{equation}\label{eq-Sf_leq_Sf0}
	S_f(t,s) 1 \leq M_\epsilon S_0(t,s) e^{\epsilon(t-s)},  \ t\geq s,
	\end{equation}
	where $M_\epsilon=e^{K_\epsilon}$.
	Hence, by the proof of Proposition \ref{prop-example-NEDII-line},
	\begin{equation*}
	S_f(t,s) 1 \leq M_\epsilon e^{-(1-\epsilon)(t-s)+2|t|}, \  t\geq s.
	\end{equation*}
	Therefore, $\mathcal{S}_f$ admits a NEDII with projections $\Pi^u(t)=0$ for every $t\in \mathbb{R}$.
	\par We now use \eqref{eq-relation-integrals-ergodic} and Proposition \ref{prop-example-NEDII-line}, with $t\geq s$ replaced by $t\leq s$, to obtain
	 that $\mathcal{S}_f$ admits a NEDII with projections $\Pi^u(t)=Id_\mathbb{R}$, for every $t\in \mathbb{R}$.
	\par Let us prove now that $\mathcal{S}_f$ does not admit NEDI. 
	Suppose that $\mathcal{S}_f$ admits a NEDI with  projections $\Pi^u(t)$ for $t\in \mathbb{R}$. Then, $\Pi^u(t)$ must be constant equal to the null operator or the identity map. First, assume that $\widetilde{\Pi}^u(t)=0$, for every $t\geq 0$. 
	Thus there exist $\widetilde{M}, \tilde{\alpha}>0$ and $\tilde{\delta}\geq0$ such that
	\begin{equation}\label{eq-Sf-NEDI-proof-prof-f}
	S_f(t,s) 1 \leq \widetilde{M} e^{-\tilde{\alpha}(t-s)+\tilde{\delta}|s|}, \ t\geq s.
	\end{equation}
	By similar arguments used to prove \eqref{eq-Sf_leq_Sf0}, it is possible to verify that
	\begin{equation}\label{eq-Sf0_leq_Sf}
	S_0(t,s) 1 \leq M_\epsilon S_f(t,s) e^{\epsilon(t-s)},  \ t\geq s.
	\end{equation}
	Then for $\epsilon\in (0,\tilde{\alpha})$, inequality \eqref{eq-Sf0_leq_Sf} implies that $\mathcal{S}_0$ admits a NEDI, which is a contradiction with Proposition \ref{prop-example-NEDII-line}. 
	By a similar analysis, it is possible to see that if $\mathcal{S}_f$ admits a NEDI with projections
	$\widetilde{\Pi}^u(t)=Id_\mathbb{R}$, for $t\in \mathbb{R}$, then $\mathcal{S}_0$ will also admits a NEDI, which will be a contradiction with Proposition \ref{prop-example-NEDII-line}. The proof is complete. 
\end{proof}

%

The next proposition provides an example of an evolution process that admits a NEDII that do not admit
any NEDI over a half-line.
\begin{proposition}
	Consider the ordinary differential 
	equation
	\begin{equation*}
	\dot{x}=g(t)x, \hbox{ for } t\geq 0,
	\end{equation*}
	where $g$ is the real function defined as
	\begin{equation*}
	g(t)= \left\{ 
	\begin{array}{l l} 
	0
	& \quad \hbox{ if } t\in (0,1],\\
	1
	& \quad \hbox{ if } t\in (n!,(n+1)!\,], \hbox{ for } n=2k, k=0,1,\cdots,
	\\ -n \, 
	& \quad \hbox{ if }  t\in (n!,(n+1)!\,], \hbox{ for } n=2k+1, k=0,1,\cdots.
	\end{array} 
	\right.
	\end{equation*} 
	Then there exists an evolution process  $\mathcal{S}_{g}=\{S_{g}(t,s): t,s\geq 0\}$ such that:
	\begin{enumerate}
		\item $\mathcal{S}_g$ admits a NEDII in $\mathbb{R}^+$ with $X^s_I(1,2)$ and projection $\Pi^u(t)=0$, $t\geq 0$.
		\item $\mathcal{S}_g$ does not admit any NEDI on $\mathbb{R}^+$.
	\end{enumerate}
\end{proposition}
\begin{proof}
	Note that
	\begin{equation*}
	S_g(t,s) \,1\leq e^{t-s}, \ \ t\geq s\geq 0.
	\end{equation*}
	Since $	t-s\leq -(t-s)+2t$, for $t\geq s\geq 0$, 
	$\mathcal{S}_g$ admits a NEDII with projection $\Pi^u=0$ and exponents
	$\alpha=1$ and $\delta=2$. 
	\par Now, we prove that $\mathcal{S}_g$ does not admit any NEDI.
	Indeed, if $\mathcal{S}_g$ admits NEDI with projection $\Pi^u(\cdot)$ constant equal to $0$ or $Id_\mathbb{R}$.
	Suppose that $\Pi^u(t)=0$, for each $t\geq 0$. This means that  
	there are 
	$M,\alpha,\delta>0$ such that 
	\begin{equation}\label{ex-eq-ned2-not-ned1}
	S_g(t,s)\,1\leq Me^{\delta|s|-\alpha(t-s)}, \hbox{ for all }
	t\geq s\geq 0.
	\end{equation}
	For $n=2k$ for some $k\in\mathbb{N}$, 
	we choose
	$t_n=(n+1)!$ and 
	$s_n=n!$, thus $t_n -s_n=ns_n$.
	Thus, from \eqref{ex-eq-ned2-not-ned1}
	\begin{equation*}
	S_g(t_n,s_n)\, 1=e^{ns_n}\leq Me^{\delta s_n-\alpha(t_n-s_n)}, \hbox{ for all even }
	n.
	\end{equation*}
	Hence, 
	\begin{equation*}
	e^{ns_n(1+\alpha)-\delta s_n}\leq M, \hbox{ for all even } n\in \mathbb{N}.
	\end{equation*}
	which is a contradiction, because the sequence on the right-hand side is not bounded.
	Therefore, $\mathcal{S}_g$ does not admit NEDI with projection $\Pi^u=0$.
	\par Now, if we assume that $\mathcal{S}_g$ admits a NEDI with projection $\Pi^u:=Id_\mathbb{R}$, following the same line of arguments above
	we will obtain a contraction. 
	Therefore,
	$\mathcal{S}_g$ does not admit any
	NEDI and the proof is complete.
\end{proof}

\subsection{NED for invertible evolution processes}
\label{subsec-NED-invertile-ep}
\par In this subsection we study nonuniform exponential dichotomies for 
invertible evolution processes. We provide a relationship between NEDI and NEDII.
As an application we establish a robustness result of NEDII.
\par Before proving the next result we recall the concept of dual operator of a linear operator in a Banach space. For an arbitrary bounded linear functional $x^*\in X^*$ we write $x^*(x):=\langle x,x^*\rangle\in \mathbb{R}$.
\begin{definition}
	Let $A:D(A)\subset X \to X$ be a linear operator such that $D(A)$ is dense in $X$. The \textbf{dual operator} $A^*:D(A^*)\subset X^*\to X^*$ of $A$ is defined by:
	$D(A^*)$ is the set of $x^*\in X^*$ such that there exists $z^*\in X^*$ such that
	\begin{equation}\label{eq-def-dual-operator}
	\langle Ax,x^*\rangle=\langle x, z^*\rangle, \  x\in D(A).
	\end{equation}
	For $x^*\in X^*$ we define $A^*x^*=z^*$ as the only element of $X^*$ that satisfies
	\eqref{eq-def-dual-operator}.
\end{definition}
\par For the next result, we only need to consider the dual operator of an bounded linear operator $A\in \mathcal{L}(X)$. Of course, 
in this situation, $D(A^*)=X^*$ and 
$A^*\in \mathcal{L}(X^*)$.
\par The next result provides a fundamental relation between these two notions of NED for 
invertible evolution processes.
\begin{theorem}\label{th-fundamental-relation-between-type1-type2}
	Let $\mathcal{S}=\{S(t,s): t,s\in \mathbb{J} \}\subset\mathcal{L}(X)$
	be an invertible evolution process in a Banach space $X$.
	Define the bounded linear operator in the dual space $X^*$
	\begin{equation*}
	T(t,s)=[S(s,t)]^*, \hbox{ for all } t,s\in \mathbb{J}.
	\end{equation*}
	Then $\mathcal{T}:=\{T(t,s): t,s\in \mathbb{J}\}$ defines a invertible 
	evolution process in $X^*$.
	\par Additionally, 
	if $\mathcal{S}$ admits a NEDI (NEDII) with bound $M(t)=Me^{\upsilon|t|}$, for $t\in \mathbb{J}$, exponent $\omega>0$, and families of projections $\Pi^u$ and $\Pi^s$, for some $M,\upsilon>0$. Then $\mathcal{T}$ admits a NEDII (NEDI) with bound $M(t)$ and exponent $\omega>0$, and family of projections $\widetilde{\Pi}^u=[\Pi^{s}]^*$ and $\widetilde{\Pi}^s=[\Pi^{u}]^*$, where
	\begin{equation}
	[\Pi^{k}]^*
	:=\{[\Pi^{k}(t)]^*: t\in \mathbb{J}\}, \ \ k=u,s.
	\end{equation}
	
\end{theorem}
\begin{proof}
	Lets first show that $\mathcal{T}$ defines an evolution process in $X^*$.
	Let $t,s,\tau\in \mathbb{J}$ then
	\begin{equation*}
	T(t,t)=[S(t,t)]^*=[Id_X]^*=Id_{X^*},
	\end{equation*}
	and also 
	\begin{equation*}
	T(t,s)T(s,\tau)=[S(s,t)]^*[S(\tau,s)]^*
	=[S(\tau,s)S(s,t)]^*=T(t,\tau),
	\end{equation*}
	where we use duality properties and that $\mathcal{S}$ is an evolution process.
	Now, let 
	$(t_n,s_n,x_n^*)$ be a sequence in $\mathbb{J}^2\times X^*$ such that
	$(t_n,s_n)x_n\to (t,s,x^*)$ as $n\to +\infty$,
	we will prove that
	$T(t_n,s_n)x_n^*\to T(t,s)x^*$
	as $n\to +\infty$.
	First, note that
	\begin{eqnarray*}
		\|T(t_n,s_n)x_n^*-T(t,s)x^*\|_{\mathcal{L}(X^*)}
		&=&
		\sup_{\|x\|_X=1}|\langle x,T(t_n,s_n)x_n^*\rangle-\langle x,T(t,s)x^*\rangle|\\
		&=&
		\sup_{\|x\|_X=1}|\langle S(s_n,t_n)x,x_n^*\rangle-\langle S(s,t)x,x^*\rangle|.
	\end{eqnarray*}
	For any $x\in X$, we have that 
	\begin{eqnarray*}
		& &|\langle S(s_n,t_n)x,x_n^*\rangle-\langle S(s,t)x,x^*\rangle|\\
		& &\leq
		|\langle S(s_n,t_n)x-S(s,t)x,x_n^*\rangle|
		+
		|\langle S(s,t)x,x_n^*-x^*\rangle|\\
		&  &\leq 
		\|x_n^*\|_{X^*}\, \|S(s_n,t_n)x-S(s,t)x\|_X
		+
		\|S(s,t)x\|_X \, \|x_n^*-x^*\|_{X^*}.
	\end{eqnarray*}
	Since $\{x_n^*\}$ is a bounded sequence in $X^*$, 
	to obtain that
	\begin{equation*}
	\lim_{n\to +\infty} \|T(t_n,s_n)x_n^*-T(t,s)x^*\|_{\mathcal{L}(X^*)}=0.
	\end{equation*}
	Therefore, $\mathcal{T}$ define an invertible evolution process in $X^*$.
	\par Now, assuming that $\mathcal{S}$ admits a NEDI and we claim that
	$\mathcal{T}$ admits a NEDII.
	\par Indeed, since $\mathcal{S}$ admits a NEDI, there exists a family of projections
	$\{\Pi^u(t)\, :\,t\in \mathbb{J} \}$ such that satisfies the conditions in Definition \ref{def-nonuniform-exp-dichotomy-1} for $\mathcal{S}$.
	\par Define $\widetilde{\Pi}^s(t):=[\Pi^u(t)]^*$ for all $t\in \mathbb{J}$. Then 
	$\{\widetilde{\Pi}^s(t):t\in \mathbb{J} \}$ is a family of projections on $X^*$ such that
	\begin{equation}
	T(t,s)\widetilde{\Pi}^s(s)
	=[S(s,t)]^*[\Pi^u(s)]^*
	=[\Pi^u(s) S(s,t)]^*,
	\end{equation}
	Since $\Pi^u(s) S(s,t)=S(s,t)\Pi^u(t)$, we conclude that
	$T(t,s)\widetilde{\Pi}^s(s) \widetilde{\Pi}^s(t)T(t,s)$. 
	\par Moreover, 
	\begin{eqnarray*}
	\|T(t,s)\widetilde{\Pi}^s(s)\|_{\mathcal{L}(X^*)}
	&=&\|[S(s,t)\Pi^u(t)]^*\|_{\mathcal{L}(X^*)}\\
	&=&\|S(s,t)\Pi^u(t)\|_{\mathcal{L}(X)}\\
	&\leq&
	Me^{\upsilon |t|} e^{\omega(s-t)}, \hbox{ for }t\geq s,
	\end{eqnarray*}
	and, if $\widetilde{\Pi}^u(t)=Id_{X^*}-\widetilde{\Pi}^s(t)$, we obtain that
	\begin{eqnarray*}
	\|T(t,s)\widetilde{\Pi}^u(s)\|_{\mathcal{L}(X^*)}
	&=&\|[S(s,t)\Pi^s(t)]^*\|_{\mathcal{L}(X^*)}\\
	&\leq&
	Me^{\upsilon |t|} e^{-\omega(s-t)}, \hbox{ for } t\leq s.
	\end{eqnarray*}
	Then, according to the inequalities above, 
	$\mathcal{T}$ admits NEDII
	on $\mathbb{J}$ with exponent $\omega>0$ and bound
	$Me^{\upsilon |t|}$, $t\in \mathbb{R}$.
	\par Finally, if $\mathcal{S}$ admits a NEDII following the same line of arguments of the above proof we conclude that
	$\mathcal{T}$ admits a NEDI, with the same relations between projections, bound and exponent.
\end{proof}

\par As an application of Theorem 
\ref{th-fundamental-relation-between-type1-type2}
we provide conditions obtain that 
NEDII is stable under perturbation. 
\par First, we borrow a robustness result for NEDI under perturbation, presented in Caraballo \textit{et al.} \cite[Theorem 3.11]{Caraballo-Carvalho-Langa-Sousa}.

\begin{theorem}[Robustness of NEDI]
	\label{th-roughness-continuous-TED}
	Let $\mathcal{S}=\{S(t,s): t\geq s\}$ be a linear evolution process which
	admits a NEDI 
	with bound $M(s)=Me^{\upsilon |s|}$ and exponent $\omega>0$ for some 
	$M>0$, and 
	$0<\upsilon<\omega$. 
	Suppose that $\mathcal{S}$ satisfies
	\begin{equation}\label{th-roughness-continuous-TED-hypothesis1}
	L_\mathcal{S}:=\sup_{0\leq t-s\leq 1} \{e^{-\upsilon |t|} \|S(t,s)\|_{\mathcal{L}(X)}\}<+\infty.
	\end{equation}
	Then there exists $\epsilon>0$ such that if 
	$\mathcal{T}=\{T(t,s)\, :\, t\geq s\}$
	is another evolution process such that
	\begin{equation}\label{th-roughness-continuous-TED-hypothesis2}
	\sup_{0\leq t-s\leq 1}\{ e^{\upsilon |s|} \|S(t,s)-T(t,s)\|_{\mathcal{L}(X)} \}<\epsilon,
	\end{equation}
	then $\mathcal{T}$ admits a NEDI with 
	exponent 
	$\hat{\omega}:=\tilde{\omega}-\omega>0$ and
	bound
	\begin{equation*}
	\widehat{M}(s)=
	\widehat{M}^2 e^{2\tilde{\omega}} \max\{L_\mathcal{T}, L_\mathcal{T}^2\}
	\,	e^{2\upsilon|s|},
	\end{equation*}
	where
	\begin{equation*}
	\tilde{\omega}=-\ln(\cosh \omega - [\cosh ^2 \omega-1-2\epsilon\sinh \omega]^{1/2}),
	\end{equation*}
	$\widehat{M}:=M(1+\epsilon/(1-\rho)(1-e^{-\omega}))\max\{M_1,M_2\}$, and
	$\rho:=\epsilon(1+e^{-\omega})/(1-e^{-\omega})$,
	$M_1:=[1-\epsilon e^{-\omega}/(1-e^{-\omega-\tilde{\omega}})]^{-1}$,
	$M_2:=[1-\epsilon e^{-\tilde{\beta}}/(1-e^{-\omega-\tilde{\beta}})]^{-1}$
	and
	$\tilde{\beta}:=\tilde{\omega}+\ln(1+2\epsilon\sinh\omega)$.
\end{theorem}
\par Now, as a consequence of Theorem \ref{th-roughness-continuous-TED} and 
Theorem \ref{th-fundamental-relation-between-type1-type2} we prove a robustness result for NEDII.
\begin{theorem}[Robustness of NEDII]
	\label{th-robustness-invertible-NED2}
	Let $\mathcal{S}_1=\{S_1(t,s): t,s\in \mathbb{R}\}$ 
	be an invertible evolution process that 
	admits a NEDII with bound $M(t)=Me^{\upsilon|t|}$, $t\in \mathbb{R}$, for some $M,\omega>0$, and exponent $\omega>\upsilon$.
	Suppose that $\mathcal{S}_1$ satisfies
	\begin{equation}\label{eq-th-robustness-invertible-NED2-1}
	\sup_{0\leq t-s\leq 1} 
	\{e^{-\upsilon|t| }\, \|S_1(t,s)\|_{\mathcal{L}(X)}\}<+\infty.
	\end{equation}
	Then there exists $\epsilon>0$ such that if $\mathcal{S}_2$
	is another invertible evolution process such that
	\begin{equation}\label{eq-th-robustness-invertible-NED2-2}
	\sup_{0\leq t-s\leq 1} 
	\{e^{\upsilon|s|}\, \|S_1(t,s)-S_2(t,s)\|_{\mathcal{L}(X)}\}<\epsilon.
	\end{equation}
	Then $\mathcal{T}_2:=\{T_2(t,s)=[S_2(s,t)]^*:t,s\in \mathbb{R} \}$ admits a NEDI with exponent
	$\hat{\omega}>0$ and bound $\widehat{M}$ provided in Theorem \ref{th-roughness-continuous-TED}
	for $\epsilon$ small enough.
	\par Additionally, if $X$ is reflexive, then $\mathcal{S}_2$ admits a NEDII with the same bound and exponent of $\mathcal{T}_2$.
\end{theorem}

\begin{proof}
	Let $\mathcal{T}_1=\{T_1(t,s): t,s\in \mathbb{R}\}$ be the evolution process over
	$X^*$ defined by 
	$T_1(t,s):=[S_1(s,t)]^*$ for all $t,s\in \mathbb{R}$.
	\par Then, from Theorem 
	\ref{th-fundamental-relation-between-type1-type2},
	$\mathcal{T}_1$ admits a NEDI with bound $M(t)=Me^{\upsilon|t|}$ and exponent  $\omega>\upsilon$.
	From \eqref{eq-th-robustness-invertible-NED2-1},
	$\mathcal{T}_1$ satisfies
	\begin{equation*}
	\sup_{0\leq t-s\leq 1} 
	\{e^{-\upsilon|t| }\, \|T_1(t,s)\|_{\mathcal{L}(X^*)}\}=
	\sup_{0\leq t-s\leq 1} 
	\{e^{-\upsilon|t| }\, \|S_1(s,t)\|_{\mathcal{L}(X)}\}<+\infty.
	\end{equation*}
	Therefore, by Theorem \ref{th-roughness-continuous-TED},
	there exists $\epsilon>0$ such that
	if $\mathcal{T}=\{T(t,s):t,s\in \mathbb{R} \}$ 
	is an evolution process over $X^*$ such that
	\begin{equation}\label{eq-proof-robustness-invertible-NED2}
	\sup_{0\leq t-s\leq 1} 
	\{e^{\upsilon|s|}\, \|T_1(t,s)-T(t,s)\|_{\mathcal{L}(X^*)}\}<\epsilon,
	\end{equation}
	then $\mathcal{T}$ admits NEDI with exponents 
	$\hat{\omega}$ and bound $\widehat{M}$, given by Theorem \ref{th-roughness-continuous-TED}.
	\par Let $\mathcal{S}_2$ be an evolution process over $X$
	such that satisfies equation
	\eqref{eq-th-robustness-invertible-NED2-2}. 
	Hence
	$T_2(t,s):=[S_2(s,t)]^*$ defines an evolution process $\mathcal{T}_2$ over
	$X^*$ satisfying
	\eqref{eq-proof-robustness-invertible-NED2}.
	Thus
	$\mathcal{T}_2$ admits a NEDI with exponent
	$\hat{\omega}>0$ and bound $\widehat{M}$, which finishes the first part of the proof.
	\par Finally, we assume that $X$ is reflexive. Let $J:X\to X^{**}$ be the evaluation map, i.e., $J$ is defined by $x\mapsto Jx\in X^{**}$, 
	where $\langle x^*,Jx(x^{**})\rangle=\langle x,x^*\rangle$, for every $x^*\in X^*$. Since $X$ is reflexive, $J$ is an isometric isomorphism, 
	and $\mathcal{S}_2$ satisfies
	\begin{equation}
	S_2(t,s)=J^{-1}[S_2(t,s)]^{**}J, \hbox{ for every }  t,s\in \mathbb{R}.
	\end{equation} 
	Now, from the first part of the proof
	 $\mathcal{S}_2^*=\{[S_2(t,s)]^*: t,s\in  \mathbb{R}\}$ admits a NEDI. Hence
	 Theorem 
	\ref{th-fundamental-relation-between-type1-type2} implies that 
	$\mathcal{S}_2^{**}=\{[S_2(t,s)]^{**}: t,s\in \mathbb{R}\}$ admits a NEDII with bound $\widehat{M}$, exponent $\hat{\omega}>0$, and family of projections $\{\widehat{\Pi}^u(t):t\in \mathbb{R}\}$. 
	Then, it is straightforward to verify that $\mathcal{S}_2$ admits a NEDII with bound $\widehat{M}$ and exponent $\hat{\omega}$, and projections
	$\Pi^u(t)=J^{-1}\widehat{\Pi}^u(t)J$ for $t\in \mathbb{R}$, and the proof is complete.
\end{proof}

\begin{remark}\label{remark-NEDII>NEDI_line}
	\par There are evolution processes such that
	has a NEDII with bound $M_2(t)=Me^{\upsilon_2|t|}$ and exponent $\omega_2>\upsilon_2$
	and admits a NEDI bound $K_1(t)=Me^{\upsilon_1|t|}$ and exponent $\omega_1<\upsilon_1$.
	For those it is possible apply the robustness result of NEDII, Theorem \ref{th-robustness-invertible-NED2}, and it is not possible to apply the robustness result of NEDI, Theorem \ref{th-roughness-continuous-TED}, because of the conditions on the exponents $\omega_1<\delta_1$, see Example \ref{example-NEDII_w2>y2_w1<y1}.
	\par Therefore, we established a robustness result of NEDII that can be applied in a situation where the robustness of NEDI, namely Theorem \ref{th-roughness-continuous-TED}, can not be applied, which reinforces the importance of the nonuniform exponential dichotomy of type II.
\end{remark}

\par Next, we provide an example of an invertible evolution process in $\mathbb{R}$, with the proprieties describe in Remark \ref{remark-NEDII>NEDI_line}.
\begin{example}\label{example-NEDII_w2>y2_w1<y1}
	Let $a,b,c,d>0$ with $b>a$ and $d>c$. 
	Consider the real function
	\begin{equation*}
	f(t)= \left\{ 
	\begin{array}{l l} 
	-b-at\sin(t), 
	&  \quad \hbox{if } t\geq 0,
	\\ -d-ct\sin(t),\, 
	& \quad \hbox{if } t< 0.
	\end{array} 
	\right.
	\end{equation*}
	Let  $\mathcal{T}=\{T(t,s):t,s\in \mathbb{R}\}$ be the evolution process induced by
	$\dot{x}(t)=f(t)x(t)$. 
	Then, from Example \ref{Example-Barreira}, 
	it is straightforward to verify that $\mathcal{T}$
	admits 
	NEDII with $X^s_{II}(\alpha_2,\delta)$ and $\Pi^u_{II}=0$,
	and
	a NEDI with $X^s_I(\alpha_1,\delta)$, and $\Pi^u_I=0$,
	where
	\begin{equation*}
	\alpha_2=\min\{b+a,d-c\}, \ 
	\alpha_1=\min\{b-a,d+c\}
	\hbox{ and }
	\delta=\max\{2a,2c\}.
	\end{equation*}
	In particular, $\mathcal{T}$ satisfies condition \ref{eq-th-robustness-invertible-NED2-1}.
	\par Note that it is possible to choose $a,b,c,d$ such that
	$\alpha_2>\delta$ and
	$\alpha_1<\delta$, for instance: $d<1/2$, $a>1$, and $b\in (a+1,3a)$.
	Thus, for these choices, it is possible to apply the Robustness of NEDII, namely Theorem \ref{th-robustness-invertible-NED2},
	and it is not possible to apply the Robustness of NEDI, Theorem \ref{th-roughness-continuous-TED}. 
	Therefore, in this case, we know for sure that NEDII is persists under perturbation and we do not know if NEDI does. 
	\par Of course, a symmetric claim holds for NEDI: there are $a,b,c,d$ such that 
	$\alpha_1>\delta$ and $\alpha_2<\delta$. Therefore, together, NEDI and NEDII, provides an completely analysis of existence of an nonuniform exponential dichotomy for $\dot{x}=f(t)x$ and whenever this type of nonuniform hyperbolicity is preserved under perturbation. 
	
\end{example}

\section{Existence of forward and pullback attractors}
\label{sec-existence-attractors}
\par In this section we study asymptotic behavior for evolution processes in a metric space $(X,d)$. The goal is to present a self contained theory
that will be appropriated for our applications on ordinary and parabolic differential equations. Hence, we study forward and pullback attraction and provide some simple results on the existence of \textit{forward attractor} and \textit{pullback attractors}.
Our results are inspired by
 Chepyzhov and Vishik \cite{Chepyzhov-Vishik} and Caraballo \textit{et al.} \cite{Caraballo-Kloeden-Real}, for forward attractions, and by 
 Carvalho \textit{et al.} \cite{Carvalho-Langa-Robison-book} and Mar\'in-Rubio and Real \cite{Marin-Real}, for pullback attraction. 
 
\par For both senses of attraction, we will need a proper definition of distance, the \textit{Hausdorff semi-distance }between sets of $X$.
\begin{definition}
	Let $A,B$ be subsets of $X$ the \textbf{Hausdorff semi-distance} between $A$ and $B$ is defined as 
	\begin{equation*}
	dist(A,B)=\sup_{a\in A}\inf_{b\in B}\,  d(a,b).
	\end{equation*}
\end{definition}
\subsection{Forward attraction}
\label{subsec-forward-attraction}
\par In this subsection we introduce some notions of forward attraction and study conditions to obtain the existence of \textit{forward attractors} for an evolution process $\mathcal{S}=\{S(t,s): t\geq s, \  t,s\in\mathbb{J}\}$ in a metric space $X$, where $\mathbb{J}$ is $\mathbb{R}$ or $\mathbb{R}^+$. 
\begin{definition}
	Let $\mathcal{S}$ be an evolution process and $\{D(t):t\in \mathbb{J}\}$ a family of nonempty subsets of $X$ and $B$ a subset of $X$. We say that $\{D(t):t\in \mathbb{J}\}$ \textbf{forward absorbs} $B$, if for any arbitrary $\tau \in \mathbb{J}$ there exists a $t_0=t_0(\tau,B)\geq \tau$ such that
	\begin{equation*}
	S(t,\tau)B\subset D(t), \ \hbox{ for every } t\geq t_0.
	\end{equation*}
	\par In particular, we say that a set $D$ \textbf{forward absorbs} $B\subset X$ if
	the single-point family $\{D(t)=D: t\in \mathbb{J}\}$ forward absorbs $B$.
\end{definition}
\par A fundamental tool on the existence of forward attractor is the notion of a \textit{forward $\omega$-limit set}.
\begin{definition}
	Let $B$ be a subset of $X$ and $\tau\in \mathbb{J}$. The \textbf{forward $\omega$-limit set of $B$ at $\tau$} is
	\begin{equation*}
	\omega_F(B,\tau)=\bigcap_{s\geq 0} \overline{\bigcup_{t\geq s} S(t+\tau,\tau) B}.
	\end{equation*}
\end{definition}
Another way to characterize a forward $\omega$-limit set is the following.
\begin{remark}
	Let $B$ be a subset of $X$ and $\tau\in \mathbb{J}$. The forward $\omega$-limit of
	$B$ at $\tau$ is 
	\begin{eqnarray*}
	\omega_F(B,\tau)=\bigg\{y\in X: \hbox{ there are sequences } \{b_n\}_{n\in \mathbb{N}}\subset B\hbox{ and } t_n\geq 0 \hbox{ with } \\t_n\to +\infty 
	 \hbox{ such that }y=\lim_{n\to +\infty}S(t_n+\tau,\tau)b_n\bigg\}.
	\end{eqnarray*}
\end{remark}
\par Now, we present the definition of \textit{forward attractor} for an evolution process $\mathcal{S}$.
\begin{definition}
	A compact set $\mathcal{A}$ on $X$ is called a \textbf{forward attractor} for $\mathcal{S}$ 
	if $\mathcal{A}$ forward attracts any bounded set in $X$, i.e., for any 
	bounded set $B\subset X$ and $\tau\in \mathbb{J}$
	\begin{equation*}
	\lim_{t\to +\infty} dist(S(t,\tau)B, \mathcal{A})=0,
	\end{equation*}  
	and $\mathcal{A}$ is the smallest closed set that forward attracts bounded sets.	
\end{definition}
\par In Chepyzhov and Vishik \cite{Chepyzhov-Vishik} it is discussed the existence of forward attractors which can be \textit{uniform} with respect to the initial time $\tau$ (\textit{w.r.t. $\tau$}). However, in our applications it is not expected to obtain this type of uniform attractors, because of the presence of nonuniform hyperbolicity that we will use in our hypotheses in the following sections.
\par Now, we are ready to present our result on the existence of forward attractors.
\begin{theorem}\label{th-existence-forward-attractor}
	Let $\mathcal{S}=\{S(t,s): t\geq s, \ t,s\in \mathbb{J}\}$ be an evolution process in $X$. 
	Suppose that there exists 
	a compact set that forward absorbs every bounded set of $X$, then there exists 
	a forward attractor for
	$\mathcal{S}$.
\end{theorem}
\begin{proof}
	\par Let $\mathcal{B}$ be the class of all bounded nonempty set of $X$ and suppose that there exists a compact set $K$ that absorbs every element of $\mathcal{B}$.
	Define  
	\begin{equation*}
	\mathcal{A}:=\overline{\bigcup_{B\in \mathcal{B}}\bigcup_{\tau\in \mathbb{J}} \omega_F(B,\tau)}.
	\end{equation*}
	\par Note that, for every $B\in \mathcal{B}$ and $\tau \in \mathbb{J}$ we have that
	$\omega_F(B,\tau)$ is the smallest closed set such that
	\begin{equation*}
	\lim_{t\to +\infty} dist(S(t+\tau,\tau)B, \omega_F(B,\tau))=0.
	\end{equation*}
	Then, for every bounded set $B$ and $\tau\in \mathbb{J}$ we have that $\omega_F(B,\tau)\subset K$. Thus, $\mathcal{A}$ a closed set contained in $K$.
	Hence $\mathcal{A}$ is a compact set which forward attracts $B$. 
	\par Moreover, let $D$ be a closed subset of $X$ that forward attract subset $B$ of $X$. Then, for each $B$ and any $\tau \in \mathbb{J}$ we have that $\omega_F(B,\tau)\subset D$. Thus $\mathcal{A}\subset D$, which means that, $\mathcal{A}$ is the minimal closed set that attracts every element of $\mathcal{B}$. Therefore
	 $\mathcal{A}$ is the forward attractor for $\mathcal{S}$.
\end{proof}
\begin{remark}
	Let $\mathcal{S}=\{S(t,s):t\geq s, \ t,s\in \mathbb{R}\}$ be an evolution process such that 
	$S(0,s)B$ is a bounded for each bounded set $B$ of $X$ and $s<0$. Then the existence of a forward attractor for $\mathcal{S}$ is completely determined by $\mathcal{S}$ restricted in $\mathbb{R}^+$. In fact, at this situation, if $\mathcal{S}_+=\{S(t,s):t\geq s\geq 0\}$ has a forward attractor $\mathcal{A}$, then $\mathcal{A}$ is the forward attractor for $\mathcal{S}$. 
\end{remark}
\subsection{Pullback attraction}
\label{subsec-existence-pullback-attractors}
\par In this subsection we study \textit{pullback asymptotic behavior}. Our goal is to provide conditions for the existence of a pullback attractor that attracts family of nonempty subsets in a \textit{universe}.
For this, we set $\mathbb{J}$ as $\mathbb{R}$ or 
$\mathbb{R}^-$. 
\par We consider the collection 
$\mathcal{M}$ consisting of all time dependent families of non-empty 
subsets of $X$, i.e., a typical element of the collection $\mathcal{M}$
is a family
\begin{equation*}
\widehat{D}:=\{D(t): D(t)\subset X, \, D(t)\neq\emptyset\, 
\hbox{ for all }t\in \mathbb{J}\}.
\end{equation*}
For $\widehat{D},\widehat{B}\in \mathcal{M}$, we write 
$\widehat{D}\subset \widehat{B}$ to mean that
$D(t)\subset B(t)$ for all $t\in \mathbb{J}$.
\begin{definition}
	A subcollection $\mathcal{D}$ of $\mathcal{M}$ 
	is said to be an \textbf{universe} if it is closed by inclusions, i.e.,
	if whenever 
	$\widehat{D}\in \mathcal{D}$ and $\widehat{B}\in \mathcal{M}$ with
	$\widehat{B}\subset \widehat{D}$, then $\widehat{B}$ also belongs to 
	$\mathcal{D}$.
\end{definition}

\par An usual example of an universe for a metric space $X$ is the collection of bounded subsets of $X$. Now, we present a nontrivial example of an universe on a Banach space $(X,\|\cdot\|_X)$.

\begin{example}\label{def-example-universe}
	For all $\gamma \geq 0$ we denote 
	$\mathcal{D}_\gamma$ a subclass of $\mathcal{M}$
	that consists of all nonempty family
	$\widehat{D}_\gamma$ for which there exists 
	$C=C(\widehat{D}_\gamma)>0$ such that 
	\begin{equation*}
	\sup_{t\in \mathbb{J}}\sup_{x_t\in D_\gamma(t)}\{e^{-\gamma|t|}\|x_t\|_X\} \leq C.
	\end{equation*}
	Note that $\mathcal{D}_0$ is the collection of the uniformly bounded families of nonempty subsets of $X$.
	\par The subclass $\mathcal{D}_\gamma$ defines an universe for each fixed $\gamma>0$ .
	In fact, 
	let $\widehat{B},\widehat{D}\in \mathcal{M}$ such that
	$\widehat{B}\subset \widehat{D}$ and $\widehat{D}\in \mathcal{D}_\gamma.$
	Since
	\begin{equation*}
	\sup_{x\in B(t)}\big\{e^{-\gamma|t|} \,\|x\|_X\big\}
	\leq 
	\sup_{x\in D(t)}\{e^{-\gamma|t|} \, \|x\|_X\}, \hbox{ for all } t\in \mathbb{J},
	\end{equation*}
	we obtain that $\widehat{B}\in \mathcal{D}_\gamma$.
\end{example}

\begin{definition}
	Let $\mathcal{S}=\{S(t,s): t\geq s\}$ be an evolution process, $\mathcal{D}$ is a universe in $X$, and $\widehat{B},\widehat{D}\in \mathcal{M}$. 
	\begin{itemize}
	\item We say that $\widehat{B}$ \textbf{pullback attracts} $\widehat{D}\in \mathcal{M}$ \textbf{at instant $t$} if
	\begin{equation*}
	\lim_{s\to-\infty}dist(S(t,s)D(s),B(t))=0.
	\end{equation*}
	\item We say that $\widehat{B}$ \textbf{pullback $\mathcal{D}$-attracts} if, 
for every
	$\widehat{D}\in \mathcal{D}$ and $t\in \mathbb{J}$, 
	$\widehat{B}$ pullback attracts $\widehat{D}\in \mathcal{D}$ at instant $t$.
\end{itemize}
\end{definition}

\begin{definition}
	Let $\mathcal{D}$ be an universe and $\mathcal{S}=\{S(t,s): t\geq s\}$ an evolution process. 
	A family of compact sets 
	$\widehat{\mathcal{A}}_{\mathcal{D}}=\{\mathcal{A}_{\mathcal{D}}(t): \ t\in \mathbb{J}\}$ is said to be a
	\textbf{pullback $\mathcal{D}$-attractor} for the
	evolution process $\mathcal{S}$ if
	\begin{enumerate}
		\item $\widehat{\mathcal{A}}_\mathcal{D}$ pullback $\mathcal{D}$-attracts;
		\item $\widehat{\mathcal{A}}_\mathcal{D}$ is \textbf{invariant}, i.e.,
		\begin{equation*}
		S(t,s)\mathcal{A}_\mathcal{D}(s)=
		\mathcal{A}_\mathcal{D}(t)S(t,s), \hbox{ for all } t\geq s;
		\end{equation*}
		\item $\widehat{\mathcal{A}}_\mathcal{D}$ is the minimal closed family that 
		pullback $\mathcal{D}$-attracts, i.e., if 
		$\widehat{F}=\{F(t): t\in \mathbb{J}\}$ is a family of closed sets that pullback $\mathcal{D}$-attracts, then $\widehat{\mathcal{A}}\subset \widehat{F}$.
	\end{enumerate}
\end{definition}
\par In some situations the pullback attractor for an evolution process can  be \textit{uniformly bounded}.
\begin{definition}
	Let $\mathcal{D}$ be an universe and $\mathcal{S}=\{S(t,s): t\geq s\}$ an evolution process, and $\widehat{\mathcal{A}}_\mathcal{D}=\{\mathcal{A}_\mathcal{D}(t):t\in \mathbb{J}\}$ a pullback $\mathcal{D}$-attractor for $\mathcal{S}$. We say that $\widehat{\mathcal{A}}_\mathcal{D}$ is \textbf{uniformly bounded} if $\cup_{t\in \mathbb{J}}\mathcal{A}_\mathcal{D}(t)$ is a bounded
	subset of $X$.
\end{definition}

\par Similar to the forward case, we also have the notion of \textit{pullback absorption} for family of subsets of $X$.

\begin{definition}
	Let $\widehat{K}=\{K(t):t\in\mathbb{J} \}$ be a family of nonempty subsets of $X$. We say
	that $\widehat{K}$ \textbf{pullback $\mathcal{D}$-absorbs} 
	if given $\widehat{D}\in \mathcal{D}$ and $t\in \mathbb{J}$ 
	the family $\widehat{K}$ pullback absorbs $\widehat{D}$ 
	in the instant $t\in \mathbb{J}$, i.e, 
	there is
	$s_0=s_0(t,\widehat{D})\leq t$ such that
	\begin{equation*}
	S(t,s)D(s)\subset K(t), \hbox{ for all } s\leq s_0.
	\end{equation*}
\end{definition}

\par To prove the existence of pullback attractor it is crucial to consider the \textit{pullback $\omega$-limit set}.

\begin{definition}
	Let $\widehat{D}\in \mathcal{M}$ and $\tau\in \mathbb{J}$. The \textbf{pullback $\omega$-limit set of $\widehat{D}$ at $\tau$} is
	\begin{equation*}
	\omega_P(\widehat{D},\tau)=\bigcap_{t\leq \tau} \overline{\bigcup_{s\leq t} S(\tau,s) D(s)}.
	\end{equation*}
\end{definition}
The pullback $\omega$-limit set can be characterized by sequences.
\begin{remark}
	Let $B$ be a subset of $X$ and $\tau\in \mathbb{J}$. The forward $\omega$-limit of
	$B$ at $\tau$ is 
	\begin{eqnarray*}
	\omega_P(\widehat{D},\tau)=\bigg\{y\in X: \hbox{ there exist sequences } s_n\leq \tau, \ s_n\to -\infty, \\
	\hbox{and } b_n\in D(s_n)   \hbox{ such that }y=\lim_{n\to +\infty}S(\tau, s_n)b_n\bigg\}.
	\end{eqnarray*}
\end{remark}

\par Finally, We state our result on the existence of pullback $\mathcal{D}$-attractors. 
\begin{theorem}\label{th-characterization-existence-pullback-attractors}
	Let $\mathcal{S}$ be an evolution process in a metric space $X$ over $\mathbb{J}$ and $\mathcal{D}$ be a universe. Suppose that there exists a family of compacts sets
	$\widehat{K}=\{K(t)\, :t\in \mathbb{J}\}$ that pullback $\mathcal{D}$-absorbs.
	Then there exists a pullback $\mathcal{D}$-attractor
	$\widehat{\mathcal{A}}_\mathcal{D}$ for the evolution process
	$\mathcal{S}$.
	Moreover, the pullback $\mathcal{D}$-attractor is $\widehat{\mathcal{A}}=\{\mathcal{A}(t): t\in\mathbb{J} \}$, where
	\begin{equation}\label{eq-definition-pullback-D-attractor}
	\mathcal{A}(t):=\overline{\bigcup_{\widehat{D}\in \mathcal{D}} \omega_P(\widehat{D},t)}, \hbox{ for each } t\in \mathbb{J}.
	\end{equation}
\end{theorem} 
\begin{proof} 
\par The proof of Theorem \ref{th-characterization-existence-pullback-attractors} follows the same line of arguments of  
	\cite[Theorem 2.12]{Carvalho-Langa-Robison-book}, some details are included for the reader convenience.
	\par Let $\widehat{K}=\{K(t)\, :t\in \mathbb{J}\}$ be the family of compact sets that pullback attracts every family of sets contained in $\mathcal{D}$. 
	For each $t\in \mathbb{J}$ define $\mathcal{A}(t)$ as in \eqref{eq-definition-pullback-D-attractor}.
	\par For each $\widehat{D}\in \mathcal{D}$ and $t\in \mathbb{J}$ we have that
	$\omega_P(\widehat{D},t)$ is the minimal closed set that pullback attracts 
	$\widehat{D}$ at instant $t$. Thus, for each $t\in \mathbb{J}$, the pullback $\omega$-limit set $\omega_P(\widehat{D},t)$ is contained in $K(t)$. Therefore, the set $\mathcal{A}(t)$ is compact, pullback attracts every $\widehat{D}\in \mathcal{D}$ at instant $t$, and it is the minimal closed family that pullback $\mathcal{D}$-attracts. Hence, to conclude that $\widehat{\mathcal{A}}$ is the pullback $\mathcal{D}$-attractor we only need to guarantee invariance. 
	\par First, let us prove invariance of pullback $\omega$-limit sets. Let $t\geq s$ with $t,s\in \mathbb{J}$, then
	\begin{equation}\label{eq-invarciance-pullback-omega-limit}
	S(t,s)\omega_P(\widehat{D},s)=\omega_P(\widehat{D},t).
	\end{equation}
	In fact, from continuity of $S(t,s)$ we obtain that $S(t,s)\omega_P(\widehat{D},s)\subset \omega_P(\widehat{D},t)$, $t\geq s$. 
	Conversely, let $x\in \omega_P(\widehat{D},t)$, then there exist $s_n\to -\infty$ as $n\to +\infty$, with $s_n\leq t$ and $x_n\in D(s_n)$ such that $x=\lim_{n\to +\infty}S(t,s_n)x_n$. 
	Since $K(s)$ pullback absorbs $D$ at $s$ and $s_n\to -\infty$, it is possible to choose $n_0>0$ such that $s_n\leq s$ and $S(s,s_n)x_n\in K(s)$ for every $n\geq n_0$. Hence 
	$y_n:=S(s,s_n)x_n$ has a convergent subsequence $\{y_{n_{j}}\}$ with limit $y\in K(s)$, in particular $\lim_j y_{n_j}=y\in \omega_P(\widehat{D},s)$. Thus, $x=S(t,s)y\in S(t,s)\omega_P(\widehat{D},s)$ which concludes the proof of \eqref{eq-invarciance-pullback-omega-limit}.
	\par Finally, we are ready to prove have that $S(t,s)\mathcal{A}(s)=\mathcal{A}(t)$.
	By continuity of operator $S(t,s)$ and invariance of the pullback $\omega$-limit sets \eqref{eq-invarciance-pullback-omega-limit}, we obtain that $S(t,s)\mathcal{A}(s)\subset\mathcal{A}(t)$. 
	Reciprocally, let $x\in \mathcal{A}(t)$, then exist $x_n\in \omega_P(\widehat{D}_n,t)$
	such that $x=\lim_{n\to +\infty}x_n$. Thanks to invariance of pullback $\omega$-limit sets \eqref{eq-invarciance-pullback-omega-limit}
	there exists $y_n\in \omega_P(\widehat{D}_n,s)$ such that $x_n=S(t,s)y_n$. 
	Since $\{y_n\}_{n\in \mathbb{N}}$ is contained in $\mathcal{A}(s)$ it is possible to extract a convergent subsequence with limit $y\in \mathcal{A}(s)$. Thus $x=S(t,s)y$ and the proof is complete.
	\end{proof}
\begin{remark}\label{remark-pullback-attractors}
	Note that Theorem \ref{th-characterization-existence-pullback-attractors} provides existence and \textbf{uniqueness} of pullback $\mathcal{D}$-attractors \textbf{without} the condition that 
	$\widehat{K}$ belongs to $\mathcal{D}$, we also \textbf{do not} have obtained (necessarily) that 
	$\widehat{\mathcal{A}}$ belongs to $\mathcal{D}$. 
	This differs from the typical results on the existence and uniqueness of pullback $\mathcal{D}$-attractors, see for instance \cite[Theorem 18]{Marin-Real}, \cite[Theorem 7]{Caraballo-Lukaszewicz-Real} or \cite[Theorem 2.50]{Carvalho-Langa-Robison-book}. 
	Usually, they suppose that 
	the evolution process $\mathcal{S}$ is 
	\textit{pullback $\mathcal{D}$-asymptotically compact} and that there exists $\widehat{B}\in \mathcal{D}$ such that pullback $\mathcal{D}$-attracts, 
	they conclude the existence of one and only one pullback $\mathcal{D}$-attractor $\widehat{A}$ that belongs to the universe $\mathcal{D}$. 
	\par In our applications, the pullback $\mathcal{D}$-attractor does not necessarily belongs to $\mathcal{D}$, see Section \ref{sec-app-odes}, for ODEs and Section \ref{sec-application-pdes}, for PDEs. 
\end{remark}
	 

\begin{remark}
	Let $\mathcal{S}=\{S(t,s): t\geq s, \ s\in \mathbb{R}\}$ is an evolution process such that $S(t,s)$ is compact for every $t>s$, then 
	the existence of a pullback $\mathcal{D}$-attractor for $\mathcal{S}$ is completely determined by the existence of a pullback $\mathcal{D}$-attractor for 
	$\mathcal{S}_-=\{S(t,s): s\leq t\leq 0\}$. 
\end{remark}

\section{Applications in nonautonomous ODEs}
\label{sec-app-odes}
\par In this section we consider a nonautonomous differential equation
\begin{equation}\label{eq-nonautonomous-problem}
\dot{x}=f(t,x), \ \ x(s)=x_s, 
\end{equation}
where $f:\mathbb{R}^{N+1} \to \mathbb{R}^{N}$ is a continuous function, locally Lipschitz in the second variable.
Then, for each $x_0\in \mathbb{R}^N$ and 
$s\in \mathbb{R}$, there exists a
solution 
$x(\cdot,s;f,x_0): [s,\sigma(s,x_0))\to \mathbb{R}^N$
of 
\eqref{eq-nonautonomous-problem}, defined on a maximal interval of existence
 $[s,\sigma(s,x_0))$ for some $\sigma(s,x_0)>s$ (possibly $+\infty$). 
If $\sigma(s,x_0)=+\infty$, then for each $x_0\in \mathbb{R}^N$ and $s\in \mathbb{R}$, this solution induces a continuous evolution process
$\{S_f(t,s)\, :\, t\geq s\}$ over $\mathbb{R}^N$ defined by
$S_f(t,s)x_0:=x(t,s;f,x_0)$.

\subsection{Comparison with a scalar ODE}
\label{subssec-comparison-scalar-odes}
\par In this subsection we develop an admissibility theory applied to asymptotic behavior of evolution processes. 
 This is done by comparison of \eqref{eq-nonautonomous-problem} with an scalar differential equation. 
 We use the ideas of admissibility for the scalar ODEs to obtain existence of pullback and forward attractors. Thanks to this technique we obtain estimates for the size of the balls that contains the attractors.
 The approach provided in this subsection can be employed in different situations. For instance, in Subsection \ref{subssec-comparison-systems-odes} for comparison with systems of ordinary differential equations, and 
  Subsection \ref{subssec-comparison-scalar-pdes} for parabolic differential equations. 
  \par The approach of this subsection is inspired by Longo \textit{et al.} \cite{Longo-Novo-Obaya}, where the authors use uniform exponential dichotomies and comparison to study asymptotic behavior for nonautonomous dynamical systems, and by Zhou and Zhang \cite{Zhou-Zhang}, where they study admissibility pairs for difference equations.
\par Consider the following assumption: 
\par \textbf{(A)} Assume that there exists $a,b:\mathbb{R}\rightarrow \mathbb{R}$ continuous real functions such that
\begin{equation*}
2\langle f(t,x), x\rangle \leq a(t)|x|^2+b(t), \hbox{ for all }
(t,x)\in \mathbb{R}^{N+1},
\end{equation*}	
where $|x|:=\sqrt{\langle x, x \rangle}$, and
$\langle \cdot, \cdot \rangle$ is the inner product on
$\mathbb{R}^N$.
\par From \textbf{(A)}, we have that 
\begin{equation*}
\frac{d}{dt}|x(t)|^2=
2\langle f(t,x(t)), x(t)\rangle
\leq a(t)|x(t)|^2+b(t), 
\hbox{ for all }
t\in \mathbb{R}.
\end{equation*}
Then, for each  $(s,x_0)\in \mathbb{R}\times \mathbb{R}^N$ the solution 
$x(t,s;f,x_0)$ is defined for every $t\geq s$, and satisfies
\begin{equation}\label{eq-nonhomogenous-perturbation-of_scalar-integral-eq}
|x(t,s;f,x_0)|^2\leq \exp\bigg\{\int_{s}^{t} a(r) dr  \bigg\}|x_0|^2 + \int_s^t 
\exp\bigg\{\int_{r}^{t} a(\tau) d\tau\bigg\} b(r) dr.
\end{equation}
\par This leads us consider the scalar nonautonomous equation
\begin{equation}\label{eq-nonhomogenous-perturbation-of_scalar-eq}
\dot{z}=a(t)z+b(t)\in \mathbb{R}.
\end{equation}
\par Let 
$T(t,s):=\exp\big\{\int_{s}^{t} a(r) dr \big\}$ for all $t,s\in \mathbb{J}$, then
$\mathcal{T}=\{T(t,s):t\geq s\}$ is the evolution process such that 
$z(t,s;a,x_0):=T(t,s)z_0$ is the solution of 
$\dot{z}=a(t)z$ with $z(0)=z_0$.
\par For each $\eta\in \mathbb{R}$ consider the space
\begin{equation*}
C_{\eta}(\mathbb{J})=\{\,
b:\mathbb{J}\rightarrow \mathbb{R}: \hbox{ b is continuous and }
\sup_{r\in \mathbb{J}} \{e^{-\eta|r|}|b(r)|\}	<+\infty\,
\}.
\end{equation*}

\par We study existence of pullback 
$\mathcal{D}_\gamma$-attractors for $\mathcal{S}_f$, by assuming that $\mathcal{T}$ admits a NED (of types I or II). 
In fact, inspired by possible admissible pairs
for \eqref{eq-nonhomogenous-perturbation-of_scalar-eq}, we will provide explicitly  radius of family of balls that contains 
the pullback attractor depending of the space that the non-homogeneous function $b$ belongs.

%
\par First, we study the existence of pullback attractors for $\mathcal{S}_f$ restricted in $\mathbb{R}^-$.
\begin{theorem}[Existence of Pullback $\mathcal{D}_\gamma$-Attractors in $\mathbb{R}^-$]
	\label{th-existence-pullback-attractors-semi-line}
	Suppose that $\mathcal{T}$ admits a NEDII on $\mathbb{R}^-$ with 
	$X_{II}^s(\alpha,\delta)$ and $\Pi^u=0$. Let $\mathcal{D}_\gamma$ be the universe defined in Example
	\ref{def-example-universe} with $\mathbb{J}=\mathbb{R}^-$,
	for every $0\leq \gamma<\alpha/2$.
	\par If $\lambda <\alpha/\delta$ and  
	$b\in C_{\lambda \delta}(\mathbb{R}^-)$,
	there exists a pullback $\mathcal{D}_\gamma$-attractor 
	$\widehat{\mathcal{A}}_\gamma:=\{A_\gamma(t): t\leq 0\}$ for $\mathcal{S}_f^-:=\{S_f(t,s): s\leq t\leq 0\}$.
	\par Moreover,
	$\mathcal{A}_\gamma(t)\subset B[0,R(t)]$, 
	for each $t\in \mathbb{R}^-$, 
	for some $R:\mathbb{R}^-\to \mathbb{R}$ which 
	$R^2\in C_{(1+\lambda)\delta}(\mathbb{R}^-)$.
	\par In particular, if $\lambda \leq -1$, the pullback $\mathcal{D}_\gamma$-attractor $\widehat{\mathcal{A}}_\gamma$ is uniformly bounded.
\end{theorem}

\begin{proof}
	We show the existence of a family of compact sets $\widehat{K}=\{K(t): t\leq 0\}$ that $\mathcal{D}_\gamma$-absorbs for every $\gamma\in [0,\alpha/2)$ fixed.
	Let $\widehat{D}\in \mathcal{D}_\gamma$,
	$b\in C_{\lambda \delta}(\mathbb{R}^-)$,
	$\lambda<\alpha/\delta$,
	and $s\leq t\leq 0$. 
	Since $\widehat{D}\in \mathcal{D}_\gamma$,
	there exists $C>0$ such that
	\begin{equation*}
	|x_s|^2\leq Ce^{2\gamma|s|}, \hbox{ for all } s\in \mathbb{R}^-,
	\hbox{ and } x_s\in D(s).
	\end{equation*}
	\par From
	\eqref{eq-nonhomogenous-perturbation-of_scalar-integral-eq}, we obtain that
	\begin{equation*}
	|S_f(t,s)x_s|^2\leq T(t,s)|x_s|^2+\int_{s}^t T(t,\tau) b(\tau) d\tau.
	\end{equation*}
	Since $\mathcal{T}$ admits a NEDII with $X^s_{II}(\alpha,\delta)$ and $\Pi^u=0$, there exist $M,\alpha>0$ and $\delta\geq 0$ such that
	\begin{equation*}
		\begin{split}
		I(t,s)&:=\int_{s}^t T(t,\tau) b(\tau) d\tau \leq
		\int_{s}^t M e^{\delta |t|} e^{-\alpha(t-\tau)}b(\tau) dt\tau\\
		&\leq
		Me^{\delta |t|} e^{-\alpha t} 
		\int_{s}^t e^{\alpha \tau + \delta \lambda |\tau|} e^{-\lambda \delta|\tau|}b(\tau) d\tau, \ \
		s\leq t\leq 0.
		\end{split}
	\end{equation*}
	Since $\alpha-\lambda \delta >0$,
	\begin{equation*}
	I(t,s)\leq 
	\frac{M}{\alpha-\delta\lambda}\|b\|_{\lambda\delta}
	\,	e^{(\lambda+1)\delta|t|}.
	\end{equation*}
	On another hand, for $0\geq t\geq s$ and 
	$x_s\in D(s)$ we obtain
	\begin{equation*}
	\begin{split}
	|T(t,s)|\, |x_s|^2&\leq 
	Me^{-(\delta +\alpha) t}e^{\alpha s}Ce^{2\gamma|s|}\\
	&\leq  MC e^{-(\delta +\alpha) t} e^{(\alpha-2\gamma)s}.
	\end{split}
	\end{equation*}
	Therefore
	\begin{equation}\label{eq-existence-pullback-scalar-comparison}
	|S_f(t,s)x_s|^2\leq M C e^{(\delta +\alpha) |t|} e^{(\alpha-2\gamma)s}+
	\frac{M}{\alpha-\delta\lambda}\|b\|_{\lambda\delta}
	\,	e^{(\lambda+1)\delta|t|}.
	\end{equation}
	Thus, for each $t\leq 0$,
	we consider the compact set 
	\begin{equation*}
	K_1(t)=B[0, R_1(t)], \hbox{ where }
	R_1(t)=\bigg\{\, 1+\frac{M}{\alpha-\delta\lambda}\|b\|_{\lambda\delta}
	\,	e^{(\lambda+1)\delta|t|\, }\bigg\}^{1/2}.
	\end{equation*}
	Hence as $s\to -\infty$ in \eqref{eq-existence-pullback-scalar-comparison} we see that 
	for each $t\leq 0$ there exists  
	$s_0=s_0(t,M,C,\lambda)\leq t$
	such that
	$S_f(t,s)x_s\in K_1(t)$ for every $s\leq s_0$ and 
	$x_s\in D(s)$.
	Now, since $s_0$ is independent of $x_s$ we obtain that
	$S_f(t,s)D(s)\subset K_1(t)$ for all
	$s\leq s_0$.
	Therefore, 
	$\widehat{K}_1:=\{K_1(t):t\in \mathbb{R}^- \}$ is a family of compact sets that
	pullback absorbs every element of the universe $\mathcal{D}_\gamma$ and from Theorem \ref{th-characterization-existence-pullback-attractors}
	there exists a pullback $\mathcal{D}_\gamma$-attractor 
	$\mathcal{A}_\gamma$ such that
	$\widehat{\mathcal{A}}_\gamma\subset \widehat{K}_1$.
	\par Now, note that from \eqref{eq-existence-pullback-scalar-comparison} 
	we conclude that 
	\begin{equation*}
	K(t)=B[0, R(t)], \hbox{ where }
	R(t)=\bigg\{\, \frac{M}{\alpha-\delta\lambda}\|b\|_{\lambda\delta}
	\,	e^{(\lambda+1)\delta|t|\, }\bigg\}^{1/2}, \ t\leq 0.
	\end{equation*}
	defines a family of compact sets $\widehat{K}:=\{K(t): t\leq 0\}$ such that 
	pullback $\mathcal{D}_\gamma$-attracts. Since $\widehat{\mathcal{A}}$ is the minimal closed family that pullback $\mathcal{D}$-attracts, we have that  
 $\widehat{\mathcal{A}}_\gamma\subset \widehat{K}$,
	and the proof is complete.
\end{proof}

\par Thanks to Theorem \ref{th-existence-pullback-attractors-semi-line}, it is possible to analyze existence of pullback attractors with an \textit{input-output} analysis, i.e., for each non-homogeneous term $b$, on the scalar equation
\eqref{eq-nonhomogenous-perturbation-of_scalar-eq},
we obtain a pullback $\mathcal{D}_\gamma$-attractor and estimates
for the size of it, which are related with the space that $b$ belongs. Note that 
the pairs $(\delta\lambda, \delta(\lambda+1))$ are similar to the admissible pairs in 
\cite{Zhou-Zhang} in the case of difference equations.


\par Now, we state a equivalent result of Theorem \ref{th-existence-pullback-attractors-semi-line}
when the evolution process $\mathcal{T}$ admits a NEDI, instead of a NEDII.

\begin{corollary}\label{corol-existence-pullback-attractors-semi-line}
	Suppose that $\mathcal{T}$ admits a NEDI on
	$\mathbb{R}^-$ with $X^s_I(\alpha,\delta)$, $\Pi^u=0$, and $\alpha>\delta$.
	If $\lambda<\alpha/\delta-1$, then the same conclusions of 
	Theorem \ref{th-existence-pullback-attractors-semi-line} hold true.
\end{corollary}

\begin{proof}
	From Theorem \ref{th-trivial-relation-between-type1-type2}, $\mathcal{T}$ admits a NEDI with
	$X_I^s(\alpha,\delta)$ with $\alpha>\delta$ and $\Pi^u =0$ if and only if $\mathcal{T}$ admits a NEDII with stable 
	set $X_{II}^s(\alpha-\delta,\delta)$ and $\Pi^u=0$. Then apply Theorem
	\ref{th-existence-pullback-attractors-semi-line}.

\end{proof}


\par The following proposition provides a relationship between the pullback attractors $\{\widehat{\mathcal{A}}_\gamma: \gamma\in [0,\alpha/2)\}$ of $\{S_f(t,s): s\leq t\leq 0\}$.
\begin{proposition}
	Let $\mathcal{D}_\gamma$ the universe defined in Example 
	\ref{def-example-universe} and 
	$\mathcal{S}_f$ be the evolution process induced by 
	\eqref{eq-nonautonomous-problem}. Suppose that
	$\mathcal{T}$ admits NEDII on $\mathbb{R}^-$ with 
	$X_{II}^s(\alpha,\delta)$ and $\Pi^u=0$.
	Define $\gamma_0:=\gamma_0(\lambda):=2^{-1}(1+\lambda)\delta$, for each 
	$\lambda\in [-1,(\alpha-\delta)/\delta)$. 
	If $b\in C_{\lambda\delta}(\mathbb{R}^-)$ the pullback
	$\mathcal{D}_{\gamma_0}$-attractor $\widehat{\mathcal{A}}_{\gamma_0}$ coincides with $\widehat{\mathcal{A}}_\gamma$ for every $\gamma\in (\gamma_0,\alpha/2)$.
	
\end{proposition}

\begin{proof}
	Theorem \ref{th-existence-pullback-attractors-semi-line} implies that for each $\lambda\in [-1, (\alpha-\delta)/\delta)$ and 
	$b\in C_{\lambda\delta}(\mathbb{R}^-)$ there exists a pullback
	$\mathcal{D}_\gamma$-attractor for every $0\leq \gamma<\alpha/2$. 
	In particular, since
	$\gamma_0\in [0,\alpha/2)$, there exists also $\widehat{\mathcal{A}}_{\gamma_0}$. 
	\par From
	Theorem
	\ref{th-existence-pullback-attractors-semi-line}, 
	for each $x_t\in \mathcal{A}_\gamma(t)$ and $t\in \mathbb{R}$, we have that
	\begin{equation*}
	|x_t|\leq R(t)=O(e^{\gamma_0|t|}).
	\end{equation*}
	Consequently, $\widehat{\mathcal{A}}_\gamma\in \mathcal{D}_{\gamma_0}$ and 
	$\widehat{\mathcal{A}}_{\gamma_0}$ pullback attracts $\widehat{\mathcal{A}}_\gamma\in \mathcal{D}_{\gamma_0}$ at instant $t\in \mathbb{R}^-$. Hence, by invariance $\widehat{\mathcal{A}}_\gamma$ of and compactness of $\mathcal{A}_{\gamma_0}(t)$ for each $t\geq 0$, we conclude that
	$\widehat{\mathcal{A}}_{\gamma}\subset \widehat{\mathcal{A}}_{\gamma_0}$, for every $\gamma\in [0,\alpha/2)$.
	\par Reciprocally, let us prove that 
	$\widehat{\mathcal{A}}_{\gamma_0}\subset\widehat{\mathcal{A}}_{\gamma}$ for all $\gamma\in  (\gamma_0,\alpha/2) $. Note that
	\begin{equation*}
	\mathcal{D}_{\gamma_0}\subset \mathcal{D}_{\gamma},\hbox{ for every } \gamma >\gamma_0.
	\end{equation*}
	\par Thus $\widehat{\mathcal{A}}_\gamma$ pullback attracts all the elements of
	$\mathcal{D}_{\gamma_0}$, and since $\widehat{\mathcal{A}}_{\gamma_0}$ is the minimal closed
	family with this property we obtain that
	$\widehat{\mathcal{A}}_{\gamma_0}\subset \widehat{\mathcal{A}}_\gamma$.
	\par Therefore, 
	$\widehat{\mathcal{A}}_\gamma=\widehat{\mathcal{A}}_{\gamma_0}$ for all $\gamma\in  (\gamma_0,\alpha/2) $.
\end{proof}

%
%

\par From now on, we are going to suppose that $\mathcal{T}$ admits exponential dichotomies with different pair of exponents on $\mathbb{R}^+$ and $\mathbb{R}^-$, because when unifying we lose quantitative information, as in Corollary \ref{corollary-trivial-information} when unifying the values of the exponents of the unstable and stable sets.
\par Next, we state a result on the forward dynamics of $\{\mathcal{S}_f(t,s):t\geq s\geq 0\}$.
\begin{theorem}[Forward Admissibility in $\mathbb{R}^+$] 
	\label{th-F-admissilibity}
	Suppose that
	$\mathcal{T}$ admits a NEDII on $\mathbb{R}^+$  
	with bound $M(t)=Me^{\nu|t|}$, exponent $\beta>0$, and $\Pi^u(t)=0$, for $t\geq 0$. Let $b\in C_{\eta\nu}(\mathbb{R}^+)$,
	 $x_0\in \mathbb{R}^N$, 
	 if
	\begin{enumerate}
		\item $\eta>-\beta/\nu$, then
		\begin{equation}\label{eq-proof-item1-F-admissibility}
		|S_f(t,s)x_0|^2\leq Me^{(\nu-\beta)t+\beta s} |x_0|^2+
		\frac{M}{\beta+\nu\eta}\|b\|_{\nu\eta}
		\,	e^{(\eta+1)\nu|t|}, \ \ t\geq s\geq 0.
		\end{equation}
		\item $\eta=-\beta/\nu$, then 
		\begin{equation}\label{eq-proof-item2-F-admissibility}
		|S_f(t,s)x_0|^2\leq Me^{(\nu-\beta)t+\beta s}|x_0|^2+M\|b\|_{\nu\eta} e^{(\nu-\beta) t}t, \ \ t\geq s\geq 0.
		\end{equation}
		\item $\eta<-\beta/\nu$, then 
		\begin{equation}\label{eq-proof-item3-F-admissibility}
		|S_f(t,s)x_0|^2\leq
		Me^{(\nu-\beta)t+\beta s} |x_0|^2-
		\frac{M\, \|b\|_{\nu\eta}}{(\beta+\eta\nu )} e^{(\nu -\beta)t}e^{\nu(\eta+1) s}, \ \ t\geq s\geq 0.
		\end{equation} 
	\end{enumerate}
\end{theorem}

\begin{proof}
	Let $b\in C_{\eta \nu}(\mathbb{R}^+)$ and $x_0\in \mathbb{R}^N$. 
	From \eqref{eq-nonhomogenous-perturbation-of_scalar-integral-eq}
	we have that 
	\begin{equation*}
	|S_f(t,s)x_0|^2\leq T(t,s)|x_0|^2+\int_{s}^t T(t,\tau) |b(\tau)| d\tau, \hbox{ for } t\geq s\geq 0.
	\end{equation*}
	Define 
	\begin{equation*}
	I(t,s):=\int_{s}^t T(t,r) |b(r)| dr, \ \ t\geq s\geq 0.
	\end{equation*}
	Since $\mathcal{T}$ admits NEDII and $b\in C_{\eta\nu}(\mathbb{R}^+)$,
	\begin{equation}\label{eq-lemma-forward-admis-I(t,s)}
	I(t,s)\leq
	Me^{-\beta t}  e^{\nu |t|} \|b\|_{\nu\eta}
	\int_{s}^t e^{(\beta  + \nu\eta )r} dr . 
	\end{equation}
	Now, the proof follows from the analysis of \ref{eq-lemma-forward-admis-I(t,s)} for each $\eta\in \mathbb{R}$.
\end{proof}

\begin{remark}
	Note that, Theorem \ref{th-F-admissilibity} includes the case where $\mathcal{T}$ admits a NEDI on $\mathbb{R}^+$  
	with bound $M(s)=Me^{\delta|s|}$, exponent $\alpha>0$, and $\Pi^u(s)=0$, for $s\geq 0$, see Theorem \ref{th-trivial-relation-between-type1-type2}.
\end{remark}
\par Now, we are ready to prove our result on the existence of forward attractors for 
$\{S_f(t,s):t\geq s\geq 0\}$. 
\begin{theorem}[Existence of Forward Attractor in $\mathbb{R}^+$]
	\label{th-existence-forward-attractor-line}
	Suppose that
	$\mathcal{T}$ admits NEDI
	on $\mathbb{R}^+$ with bound $M(t)=Me^{\delta|t|}$, $t\geq 0$, for some $M>0$ and $\delta\geq 0$, exponent $\alpha>0$, and projections $\Pi^u(s)=0$ for all $s\geq 0$.
	Then, if 
	$b\in C_{\eta\delta}(\mathbb{R}^+)$ and 
	$\eta\leq -1$, there exists a forward attractor $\mathcal{A}$ for 
	 $\{S_f(t,s):\, t\geq s\geq 0\}$. 
	 \par More Precisely, 
	  for every $\eta<-1$, the set
	 $\mathcal{A}=\{0\}$ is the forward attractor for $\{S_f(t,s):\, t\geq s\geq 0\}$, and 
	 for $b\in C_{-\delta}(\mathbb{R}^+)$, or equivalently $\eta =-1$, 
	\begin{equation*}
	\mathcal{A}\subset B[0,R], \hbox{ where } R=(\,M\alpha^{-1}\|b\|_{-\delta}\,)^{1/2}.
	\end{equation*}
\end{theorem}

\begin{proof}
	Thanks to Theorem \ref{th-trivial-relation-between-type1-type2} we see that 
	$\mathcal{T}$ admits a NEDII on $\mathbb{R}^+$ with $X_{II}^s(\alpha+\delta,\delta)$ and $\Pi^u(s)=0$. Hence, $\mathcal{T}$ satisfies the hypothesis of Theorem \ref{th-F-admissilibity} with $\beta=\alpha+\delta>\delta=\nu$.
	 \par If $b\in C_{\eta\delta}(\mathbb{R}^+)$ with $\eta<-1$, from the proof of Theorem \ref{th-F-admissilibity}, it is straightforward to verify that
	every neighborhood of $\{0\}$ forward absorbs every bounded subset of $\mathbb{R}^N$. Hence $\{0\}$ forward attracts every bounded set. 
	Thus, from Theorem \ref{th-existence-forward-attractor}, there exists a forward attractor 
	$\mathcal{A}$ of $\{S_f(t,s): t\geq s\geq 0\}$. 
	Moreover, since $\{0\}$ is a closed set that forward attracts bounded sets, we conclude that $\{0\}$ is the forward attractor of
	$\{S_f(t,s):t\geq s\geq 0\}$.
	\par Finally, let $b\in C_{-\delta}(\mathbb{R}^+)$ ($\eta=-1$). 
	From \eqref{eq-proof-item1-F-admissibility} we have that 
	the closed set 
	$K=B[0,R]$, with $R=(\,M\alpha^{-1}\|b\|_{-\delta}\,)^{1/2}$, forward attracts each bounded set of $\mathbb{R}^N$. Therefore,
	 from Theorem \ref{th-existence-forward-attractor}, there exists a forward attractor 
	$\mathcal{A}\subset K$, and the proof is complete.
%
\end{proof}

\begin{remark}
	Under the same conditions, Theorem \ref{th-existence-forward-attractor-line} 
	holds true for $\mathcal{S}_f=\{S_f(t,s):\, t\geq s, \ t,s\in \mathbb{R} \}$.
	In fact, since $S(t,s)$ is a compact operator for every $t>s$, the forward attractor for $\mathcal{S}_f$ is the same forward attractor of $\{S_f(t,s): t\geq s\geq 0\}$. 
\end{remark}
%
\par Now, we state a general result about existence of pullback attractors for $\mathcal{S}_f$.

\begin{theorem}[Existence of Pullback
	$\mathcal{D}_\gamma$-attractor in $\mathbb{R}$]
	\label{th-existence-pullback-attractors-line}
	Let $\mathcal{D}_\gamma$ the universe defined on 
	Example \ref{def-example-universe} and 
	$\mathcal{S}_f=\{S_f(t,s):t\geq s\}$ 
	be the evolution process induced by 
	\eqref{eq-nonautonomous-problem}. Suppose that
	$\mathcal{T}$ admits NEDII on $\mathbb{R}^-$ with $X^s_{II}(\alpha,\delta)$, and
	NEDII
	on $\mathbb{R}^+$ with $X^s_{II}(\beta,\nu)$, and both with null unstable sets, i.e., 
	$\Pi^u_+(t)=0=\Pi^u_-(s)$ for $t\geq 0\geq s$.
	\par Suppose that $\lambda<\alpha/\delta$, then
	for a continuous function 
	$b:\mathbb{R}\to \mathbb{R}$ such that
	$b^-:=b|_{\mathbb{R}^-}\in C_{\lambda\delta}(\mathbb{R}^-)$
	and
	$b^+:=b|_{\mathbb{R}^+}\in C_{\eta\nu}(\mathbb{R}^+)$,
	there is a pullback $\mathcal{D}_\gamma$-attractor 
	$\widehat{\mathcal{A}}_\gamma=\{\mathcal{A}_\gamma(t): t\in \mathbb{R} \}$ for
	$\mathcal{S}_f$, for every $\gamma\in [0,\alpha/2)$.
	\par Moreover,
	$\mathcal{A}_\gamma(t)\subset B[0,R(t)]$, 
	for each $t\in \mathbb{R}$, for some 
	$R:\mathbb{R}\to \mathbb{R}$ such that
	$(R|_{\mathbb{R}^-})^2\in C_{(1+\lambda)\delta}(\mathbb{R}^-)$ and
	\begin{enumerate}
		\item If $\eta>-\beta/\nu$, then $(R|_{\mathbb{R}^-_*})^2\in C_{(1+\eta)\nu}(\mathbb{R}^-_*)$.
		\item If $\eta<-\beta/\nu$ then $(R|_{\mathbb{R}^-_*})^2\in C_{(\nu-\beta)}(\mathbb{R}^-_*)$.
		\item If $\eta=-\beta\nu$, then 
		$(R|_{\mathbb{R}^-_*})^2\in C_{(\nu-\beta)+\epsilon}(\mathbb{R}^-_*)$, for every $\epsilon>0$,
	\end{enumerate} 
where $\mathbb{R}^-_*=\mathbb{R}^+\setminus \{0\}$.
\end{theorem}

\begin{proof}
	\par We prove that there is a family of compact sets $\widehat{K}$ such that $\mathcal{D}_\gamma$-absorbs for every $\gamma\in [0,\alpha/2)$ fixed.
	Let $\widehat{D}$ be a family of subsets in  $\mathcal{D}_\gamma$, then
	there exists $C>0$ such that
	\begin{equation*}
	|x_s|^2\leq Ce^{2\gamma|s|}, \hbox{ for all } s\in \mathbb{R}
	\hbox{ and } x_s\in D(s).
	\end{equation*}
	
	\par From \eqref{eq-nonhomogenous-perturbation-of_scalar-integral-eq}
	we have for all $t\geq s$ that
	\begin{equation*}
	|S_f(t,s)x_s|^2\leq T(t,s)|x_s|^2+\int_{s}^t T(t,\tau) b(\tau) d\tau.
	\end{equation*}
	Thanks to Theorem \ref{th-existence-pullback-attractors-semi-line}, there exists 
	a pullback $\mathcal{D}$-attractor
	$\widehat{\mathcal{A}}_\gamma^-=\{\mathcal{A}_\gamma^-(t):\,t\leq 0\}$ such that
	$\mathcal{A}_\gamma^-(t)\subset B[0,R_-(t)]$, for all $t\leq 0$, where 
	$R_-(t)^2=O(e^{(\lambda+1)\delta|t|})$, for every $t\leq 0$.
	\par Therefore, to guarantee the existence of a pullback $\mathcal{D}_\gamma$-attractor for $\mathcal{S}_f$ defined for all $t\geq 0$ it remains to prove the existence of a family of compact sets defined for each $t>0$ that pullback $\mathcal{D}_\gamma$-absorbs on $X$. 
	\par We will prove just the first item, the proof of other items is similar. 
	Suppose
	that $\eta>-\beta/\nu$ and let $s\leq 0\leq t$, 
	from the proofs of Theorem 
	\ref{th-existence-pullback-attractors-semi-line}
	and
	Theorem \ref{th-F-admissilibity} we have that 
	\begin{equation*}
	I(t,s)\leq I(t,0)+I(0,s)\leq 
	\frac{M\|b^+\|_{\eta\nu}}{\beta+\eta\nu} 
	e^{(\eta+1)\nu|t|}+
	\frac{M\|b^-\|_{\delta\lambda}}{\alpha-\lambda\delta},
	\end{equation*}
	Also note that
	\begin{eqnarray*}
		T(t,s)|x_s|^2&\leq& T(t,0) T(0,s)|x_s|^2\\
		&\leq&
		Me^{\nu |t|-\beta t} \, Me^{\alpha s} Ce^{2\gamma|s|}\\
		&=&
		C M^2e^{(\nu -\beta) t}\, e^{(\alpha-2\gamma) s}.
	\end{eqnarray*}
	Hence, 
	\begin{equation}\label{eq-existence-pullbackattractor-line}
	|S_f(t,s)x_s|^2\leq C M^2e^{(\nu -\beta) t}\, e^{(\alpha-2\gamma) s}+
	\frac{M\|b^+\|_{\eta\nu}}{\beta+\eta\nu} 
	e^{(\eta+1)\nu|t|}+
	\frac{M\|b^-\|_{\delta\lambda}}{\alpha-\lambda\delta}.
	\end{equation}
	\par Now, for each $t>0$, define the following compact set 
	\begin{equation*}
	J(t):=B[0,R_+(t)] \hbox{ and }R_+(t):=\bigg(\frac{M\|b^+\|_{\eta\nu}}{\beta+\eta\nu} 
	e^{(\eta+1)\nu|t|}+
	\frac{M\|b^-\|_{\delta\lambda}}{\alpha-\lambda\delta}\bigg)^{1/2}.
	\end{equation*}
	Thus, by similar arguments of the proof of Theorem \ref{th-existence-pullback-attractors-semi-line} we see that 
	\eqref{eq-existence-pullbackattractor-line} implies 
	that $\{J(t):t> 0\}$ is a family of compact sets
	that $\mathcal{D}_\gamma$-absorbs. 
	Thus by Theorem \ref{th-characterization-existence-pullback-attractors}
	there is a pullback $\mathcal{D}_\gamma$-attractor 
	$\widehat{\mathcal{A}}_\gamma=\{\mathcal{A}_\gamma(t): t \in \mathbb{R}\}$
	such that
	$\mathcal{A_\gamma}(t)\subset J(t)$ for every $t> 0$ and 
	$\mathcal{A}_\gamma(t)\subset B[0,R_-(t)]$ for $t\leq 0$.
	Consequently, we ensure the existence of a pullback $\mathcal{D}_\gamma$-attractor
	$\widehat{\mathcal{A}}_\gamma$ with 
	$\mathcal{A}_\gamma(t)\subset B[0,R(t)]$ for each $t\in \mathbb{R}$, where $R:\mathbb{R}\to \mathbb{R}$ defined as $R(t):=R_+(t)$, for $t\geq 0$, and 
	$R(t):=R_-(t)$, for $t<0$.
\end{proof}
\begin{remark}
	Another idea to prove Theorem \ref{th-existence-forward-attractor-line} is to propagate the pullback attractor $\{\mathcal{A}_\gamma(t)^-:t\leq 0\}$ to the entire line by
	$\mathcal{A}_\gamma(t)=S(t,0)\mathcal{A}^-_\gamma(0)$, for every $t>0$, and use Theorem \ref{th-F-admissilibity}
	to conclude the above estimates for $\mathcal{A}(t)$ for $t\geq 0$.
\end{remark}

\par Next, we summarize Theorem \ref{th-existence-forward-attractor-line} and Theorem \ref{th-existence-pullback-attractors-line} with simpler hypotheses on $\mathcal{T}=\{T(t,s): t, s\in \mathbb{R} \}$.

\begin{theorem}[Existence of Pullback and Forward Attractor]
	\label{th-existence-pullback-attractors-line-verion-2}
	Let $\mathcal{S}_f$ be the evolution process induced by 
	\eqref{eq-nonautonomous-problem} and $\mathcal{D}_\gamma$ the universe defined in 
	Example \ref{def-example-universe}. Suppose that
	$\mathcal{T}$ admits NEDII on $\mathbb{R}$ with 
	$X^s_{II}(\alpha,\delta)$ and $\Pi^u(t)=0$, $t\in \mathbb{R}$.
	\par For $\lambda\in(-\alpha/\delta,\alpha/\delta)$ and 
	$b\in C_{\lambda \delta}(\mathbb{R})$
	there exists a pullback $\mathcal{D}_\gamma$-attractor $\widehat{\mathcal{A}}_\gamma=\{\mathcal{A}_\gamma(t): t\in \mathbb{R}\}$ 
	for $\mathcal{S}_f$, for every
	$\gamma<\alpha/2$. 
	Moreover,
	$\mathcal{A}_\gamma(t)\subset B[0,R(t)]$ for all $t\in \mathbb{R}$, where 
	$R^2\in C_{(1+\lambda)\delta}(\mathbb{R}) $. 
	\par Furthermore, if
	$\alpha>\delta$ and $\lambda\in (-\alpha/\delta,-1]$, then:
	\begin{itemize}
		\item The pullback $\mathcal{D}_\gamma$-attractor $\widehat{\mathcal{A}}_\gamma$ is uniformly bounded.
		\item There exists a forward attractor $\mathcal{A}_F$ for $\mathcal{S}_f$, 
		with
		\begin{eqnarray*}
		\mathcal{A}_F=\{0\},   & &\hbox{for} -\alpha/\delta<\lambda <-1,\hbox{ and }\\
		\mathcal{A}_F\subset B[0,R_F], & & \hbox{ for } \lambda=-1,
		\end{eqnarray*}
		where $R_F=[\,M\|b\|_{-\delta}/(\alpha-\delta)\,]^{1/2}$.
	\end{itemize}
\end{theorem}

\begin{remark}
	The admissibility relation is the following: for $b(t)$ growing as
	$e^{\lambda \delta|t|}$, the ``square'' of the solutions norms will be of order 
	$e^{(\lambda +1)\delta|t|}$. In particular, 
	when $\lambda=0$, $b$ is bounded, it is not expected to obtain a uniformly bounded pullback $\mathcal{D}_\gamma$-attractor.
	This happens when $b$ is exponentially small, for instance, when $\lambda=-1$ and $b(t)=O(e^{-\delta|t|})$, which is coherent with \cite[Theorem 1]{Zhou-Zhang} for discrete dynamics, but now the admissible pairs provide a interpretation on the asymptotic behavior of continuous evolution processes.
\end{remark}

%
%

\subsection{Comparison with a system of linear equations}
\label{subssec-comparison-systems-odes}

\par In this subsection we use a system of linear differential equations on $\mathbb{R}^N$ in order to compare with
Equation \eqref{eq-nonautonomous-problem}, and we obtain similar results to the ones obtained in Subsection \ref{subssec-comparison-scalar-odes}.
\par In the following, for every 
$i=1,\cdots,N$ the $i$-th component of a vector $x\in \mathbb{R}^N$ will be denoted 
by $x_i$.
Moreover, if we write $x\geq 0$ we mean that 
$x_i\geq 0$ for all
$i\in \{1,\cdots,N\}$, whereas 
$x\gg 0$ if all $i$-th components are positive, i.e., $x_i>0$.
The subspace $\mathbb{R}^N_+$ will denote all the non-negative vectors of $\mathbb{R}^N$, i.e., $x\in\mathbb{R}^N_+$ if and only if $x\geq 0$.
\par Also for this subsection we consider the following conditions:
	\par \textbf{(B1)} If $x\geq 0$ with $x_i=0$ then $f_i(t,x)\geq 0$ for all $t\in \mathbb{R}$;
	\par \textbf{(B2)} For all $(t,x)\in \mathbb{R}\times( \mathbb{R}^{N})^+$
	\begin{equation*}
	f(t,x)\leq A(t)x + b(t),
	\end{equation*}
	where $t\mapsto A(t)\in \mathcal{L}(\mathbb{R}^N)$ is a continuous function such that
	$A(\cdot)=[a_{ij(\cdot)}]\geq 0$ for every $i\neq j$ and
	$b:\mathbb{R}\to \mathbb{R}^N_+$ is also continuous.
	
\smallskip
	\par Let $\mathcal{T}:=\{T(t,s): t\geq s\}$ be the linear evolution proess over $\mathbb{R}^N$ induced by $\dot{y}=A(t)y$.
\medskip
\par Assumptions \textbf{(B1)} and \textbf{(B2)} imply that 
for each $x_0\in \mathbb{R}^N_+$ and $s\in \mathbb{R}$
the solution $x(t)=x(t,s;f,x_0)$ of \eqref{eq-nonautonomous-problem} is defined 
for every $t\geq s$ and satisfies
\begin{equation}
0\leq x(t)\leq y(t), \ \ t\geq s,
\end{equation}
  where $y(t)=y(t,s;x_0)$ is the solution of $\dot{y}=A(t)y+b(t)$ with initial data 
$y(s)=x_0$. 
Hence
\begin{equation}\label{eq-cooperatives}
0\leq x(t)\leq T(t,s)x_s+\int_s^tT(t,\tau)b(\tau)d\tau, \ \ t\geq s\geq 0.
\end{equation}
 Therefore, the
unique solution $x(\cdot,s;f,x_0):[s,+\infty)\to \mathbb{R}^N$ of Problem
\eqref{eq-nonautonomous-problem},
is globally defined, and it is
associated with the evolution process
$\{S_f(t,s): t\geq s\}$ over $\mathbb{R}^N_+$,
defined by $S_f(t,s)x_0:=x(t,s;f,x_0)\in \mathbb{R}^N_+$ for each $x_0\in \mathbb{R}^N_+$ and $t\geq s\in \mathbb{R}$. 
\par Now, we study the asymptotic behavior of the evolution process $\mathcal{S}_f:=\{S_f(t,s):t\geq s\}$ defined over the metric space $(\mathbb{R}^N_+,d)$, where $d(x,y)=|x-y|$, for $x,y\in \mathbb{R}^N_+$, and $|x|=\max_{i=1,\cdots,N}\{|x_i|\}$. Finally, we write
$B^+[0,R]:=\mathbb{R}^N_+\cap B[0,R]$, for any $R>0$.
\par Define the following space
\begin{equation*}
C_\eta(\mathbb{J},\mathbb{R}^N_+):=\bigg\{b:\mathbb{J}\to \mathbb{R}^n_+: \sup_{r\in \mathbb{R}}e^{-\eta|r|}\, |b(r)|\bigg\},
\end{equation*}
where $\mathbb{J}=\mathbb{R}, \mathbb{R}^-,\mathbb{R}^+.$ 

\par First, we state our result on the existence of Pullback $\mathcal{D}_\gamma$-attractors.

\begin{theorem}[Existence of Pullback
	$\mathcal{D}_\gamma$-attractor]
	\label{th-existence-pullback-attractors-line-cooperative}
	Let $\mathcal{S}_f$ and $\mathcal{T}$ be the evolution processes defined above. 
	Suppose that
	$\mathcal{T}$ admits NEDII on $\mathbb{R}^-$ with 
	$X^s_{II}(\alpha,\delta)$, and
	NEDII
	on $\mathbb{R}^+$ with $X^s_{II}(\beta,\nu)$, and 
	both with null unstable projections, i.e.,
	$\Pi^u_+(t)=0=\Pi^u_-(s)$, $t\geq 0\geq s$. 
	Consider $\mathcal{D}_\gamma$ the universe defined on Example
	\ref{def-example-universe} for every $\gamma\in [0,\alpha)$.
	\par Suppose that $\lambda<\alpha/\delta$, then
	for a continuous function 
	$b:\mathbb{R}\to \mathbb{R}^N_+$ such that
	$b^-:=b|_{\mathbb{R}^-}\in C_{\lambda\delta}(\mathbb{R}^-,\mathbb{R}^N_+)$
	and
	$b^+:=b|_{\mathbb{R}^+}\in C_{\eta\nu}(\mathbb{R}^+,\mathbb{R}^N_+)$,
	there is a pullback $\mathcal{D}_\gamma$-attractor 
	$\widehat{\mathcal{A}}_\gamma=\{\mathcal{A}_\gamma(t):t\in \mathbb{R}\}$ for
	$\mathcal{S}_f$, for every $\gamma\in [0,\alpha)$.
	\par Moreover,
	$\mathcal{A}_\gamma(t)\subset B^+[0,R(t)]$, 
	for each $t\in \mathbb{R}$, where 
	$R|_{\mathbb{R}^-}\in C_{(1+\lambda)\delta}(\mathbb{R}^-)$ and
	\begin{enumerate}
		\item If $\eta>-\beta/\nu$, then $R|_{\mathbb{R}^-_*}\in C_{(1+\eta)\nu}(\mathbb{R}^-_*)$.
		\item If $\eta<-\beta/\nu$ then $R|_{\mathbb{R}^-_*}\in C_{(\nu-\beta)}(\mathbb{R}^-_*)$.
		\item If $\eta=-\beta\nu$, then $R|_{\mathbb{R}^-_*}\in C_{(\nu-\beta)+\epsilon}(\mathbb{R}^-_*)$, for every $\epsilon>0$.
	\end{enumerate} 
\end{theorem}

\begin{proof}
	Let $\widehat{D}\in \mathcal{D}_\gamma$,
	$b\in C_{\lambda \delta}(\mathbb{R},\mathbb{R}^N_+)$ 
	and $s\leq t\leq 0$. 
	Let $x_s\in D(s)$,
	Since $\widehat{D}\in \mathcal{D}_\gamma$, there exists $C>0$ such that
	$|x_s|\leq Ce^{\gamma|s|}$, for every $s\in \mathbb{J}$. 
	Thus, \eqref{eq-cooperatives} yields to
	\begin{eqnarray*}
		| S_f(t,s)x_s|&\leq &|T(t,s)x_s|+\int_s^t|T(t,\tau)|\,|b(\tau)|d\tau\\
		&\leq& CMe^{\delta|t|-\alpha t} e^{(\gamma-\alpha)s}+Me^{\delta|t|-\alpha t}
		\int_{s}^{t} e^{\alpha \tau} |b(\tau)|d\tau.
	\end{eqnarray*}
	
	\par Now, following the same line of arguments of Theorem \ref{th-existence-pullback-attractors-line}, we 
	prove that there is a family of compact sets $\widehat{K}$ such that $\mathcal{D}_\gamma$-absorbs for every $\gamma\in [0,\alpha)$ fixed.
\end{proof}

\begin{remark}
	Every result of Subsection \ref{subssec-comparison-scalar-odes} could be prove for 
	$S_f(t,s):\mathbb{R}^N_+\to \mathbb{R}^N_+$ by following the same line of arguments of Theorem \ref{th-existence-pullback-attractors-line-cooperative}. 
	\par There are two main differences:
	\begin{enumerate}
		\item for every $\gamma\in (0,\alpha]$ there exist a pullback $\mathcal{D}\gamma$-attractor;
		\item the admissibility $\lambda\delta$ to $(\lambda+1)\delta$, is provide by $b$ to $R$ on $\mathbb{R}^-$, not by $b$ to $R^2$ as in Subsection \ref{subssec-comparison-scalar-odes}.
	\end{enumerate}
\end{remark}


\begin{remark}
	Condition \textbf{(B2)} implies that the systems 
	$\dot{y}=A(t)y+b(t)$ is quasi monotone. This hypothesis could be replaced by considering 
	$\dot{x}=f(t,x)$ as quasi monotone, i.e., 
	\begin{equation*}
	f_i(t,x)\leq f_j(t,z), \hbox{ whenever } x\leq z \hbox{ and } x_i=z_i.
	\end{equation*}
\end{remark}

\par To finish this subsection, we state a result on the relation between the pullback attractors $\{\widehat{\mathcal{A}}_\gamma: \gamma\in [0,\alpha)\}$.
\begin{proposition}
	Let $\mathcal{D}_\gamma$ the universe defined in Example 
	\ref{def-example-universe} and 
	$\mathcal{S}_f$ be the evolution process induced by 
	\eqref{eq-nonautonomous-problem}. Suppose that
	$\mathcal{T}$ admits NEDII on $\mathbb{R}^-$ with 
	$X_{II}^s(\alpha,\delta)$ and $\Pi^u=0$.
	Define $\gamma_0:=\gamma_0(\lambda):=(1+\lambda)\delta$, for each 
	$\lambda\in [-1,(\alpha-\delta)/\delta)$. 
	If $b\in C_{\lambda\delta}(\mathbb{R}^-)$ the pullback
	$\mathcal{D}_{\gamma_0}$-attractor $\widehat{\mathcal{A}}_{\gamma_0}$ coincides with $\widehat{\mathcal{A}}_\gamma$ for every $\gamma\in (\gamma_0,\alpha)$.
	
\end{proposition}

\section{Applications in nonautonomous parabolic PDEs}
\label{sec-application-pdes}
\par In this section we consider parabolic partial differential equation. We first present a general setting for
reaction-diffusion equation, with Dirichlet, Neumann or Robin boundary conditions. 
Then we study asymptotic behavior for these equations via comparison results, using the same approach of Section \ref{sec-app-odes} for ODEs.
Then, we study existence of nonuniform exponential dichotomies using the examples of Subsection \ref{subsec-examples_NEDII}.  Finally, consider the adjoint problem to study the relation between NEDI and NEDII. 
\vspace{0.3cm}
\par  Let us consider a scalar parabolic PDE
\begin{equation}
\left\{ 
\begin{array}{l l} 
v_t=\Delta v+ f(t,x,v),& \ t>s, \  x\in \Omega,\\
v(s)=v_0,&  \ x\in \Omega,\\
B(v):=\alpha(x)v+\kappa\frac{\partial v}{\partial n}=0,& \  t\geq s, \ x\in \partial \Omega. 
\end{array} 
\right.
\end{equation}
where $\Omega$ is a open, bounded, and connected subset of $\mathbb{R}^N$ ($N\geq 1$) with a smooth boundary $\partial \Omega$; $\Delta$ is the Laplacian operator; 
$\partial/\partial n$ denotes the outward normal derivative at the boundary; 
 and $f$ satisfies the following conditions:
\par \textbf{(H)} $f:\mathbb{R}\times \overline{\Omega}\times \mathbb{R}\to \mathbb{R}$ is continuous; and the mapping 
$\mathbb{R}\ni v\mapsto f(t,x,v)$ is Lipschitz in bounded sets uniformly for $x\in \overline{\Omega}$.
\par The problem has Dirichlet boundary conditions if $\kappa=0$ and $\alpha(x)\equiv1$; Neumann boundary conditions if $\kappa=1$ and $\alpha\equiv 0$; and Robin boundary conditions if
$\kappa=1$ and $\alpha:\partial \Omega\to \mathbb{R}$ is a nonnegative sufficiently regular function.
\par \textbf{The case of Neumann or Robin boundary conditions:} we consider the Banach space $Y:=C(\overline{\Omega})$ of continuous functions on $\overline{\Omega}$ with the sup-norm $\|\cdot\|_Y$. Let $A$ be the closure of the differential operator $A_0:D(A_0)\subset Y\to Y$, $A_0u=\Delta u$, defined on 
\begin{equation*}
D(A_0):=\{u\in C^2(\Omega)\cap C^1(\overline{\Omega}) : \, A_0u\in C(\overline{\Omega}), \ Bu=0, \hbox{ on }\partial \Omega  \}.
\end{equation*}
\par \textbf{The case of Dirichlet boundary conditions:} we consider the Banach space $C_0(\overline{\Omega})$ of the continuous maps vanishing on the boundary $\partial \Omega$ with the sup-norm $\|\cdot\|_Y$, trying to keep a common notation for all the types of boundary conditions. 
Consider the operator
$A_0:D(A_0)\subset Y\to Y$, $A_0u=\Delta u$, defined on 
\begin{equation*}
D(A_0):=\{u\in C^2(\Omega)\cap C_0(\overline{\Omega}) : \, A_0u\in C_0(\overline{\Omega})\} 
\end{equation*}
which is closable in $C_0(\overline{\Omega})$, and let $A$ be its closure. 

\smallskip

\par Thus, for all types of boundary conditions describe above, $A$ generates an analytic semigroup $\{e^{At}: t\geq 0\}$ on $Y$. Moreover, $e^{At}$ is a compact operator for each $t>0$.
\par Define $f^e:\mathbb{R}\times Y\to Y$, $(t,u)\mapsto f^e(t,u)$, 
$f^e(t,u)(x):=f(t,x,u(x))$, for each $x\in \overline{\Omega}$, so the regularity conditions
\textbf{(H)} on $f$ are transfer to $f^e$. 
Then we obtain an abstract Cauchy problem in $X$
\begin{equation}\label{th-abstract-Cauchy-problem}
\left\{ 
\begin{array}{l l} 
\dot{u}=A u+ f^e(t,u),& \ t>s,\\
u(s)=u_0.&
\end{array} 
\right.
\end{equation}
Hence Problem \eqref{th-abstract-Cauchy-problem} has a unique \textit{mild solution}, i.e., for each $[s,u_0)\in \mathbb{R} \times Y$ there exists a unique
map $u(t)=u(t,s;u_0)$ defined on a maximal interval $[s,\sigma)$, for some $\sigma(s,u_0)>s$ (possibly $+\infty$) which satisfies the integral equation 
\begin{equation*}
u(t)=e^{A(t-s)}u_0+\int_{s}^{t} e^{A(t-\tau)} f^e(\tau,u(\tau))\, d\tau, \ t\in [s,\sigma).
\end{equation*}
Moreover, $u(\cdot,s;u_0):(s,\sigma(s,u_0))\to Y$ is continuous, and $u(t)\to u_0$, as $t\to s^+$ if and only if $u_0\in \overline{D(A)}$, see Lorenzi \textit{et al.} \cite{Lorenzi-Lunardi-Metafune-Pallara}.
\par Now, define $X:=\overline{D(A)}\subset Y$, which is a Banach space with the induced norm of $Y$. For Neumann or Robin boundary conditions the phase space is $X=C(\overline{\Omega})=Y$, and for Dirichlet boundary conditions we choose $X=C_0(\overline{\Omega})$. Then, in both situations, the mild solution $u(\cdot,s;u_0):[s,\sigma(s,u_0))\to X$ is continuous.
\par Moreover, if a mild solution remains bounded, then it is defined for every 
$t\in [s,+\infty)$. 
Then, the mild solution $u$ generates a nonlinear evolution process 
$\mathcal{S}_f:=\{S_f(t,s): t\geq s\}$ over $X$, defined by $S_f(t,s)u_0=u(t,s;u_0)$.
Additionally, $\{S_f(t,s): t> s\}$ is a family of compact operators over $X$, see
Travis and Webb \cite{Travis-Webb}.

\subsection{Asymptotic behavior for parabolic PDEs}
\label{subssec-comparison-scalar-pdes}
\par Following the ideas of Section \ref{sec-app-odes}, we study pullback and forward dynamics for parabolic partial differential equations using nonuniform exponential dichotomies and comparison methods.

\par First, we recall some basic notions of monotone semi-flows. 
\begin{definition}
	We say that $X$ is a \textbf{ordered Banach space} if there is a closed convex \textit{cone of 
		nonnegative vectors} $X_+$ where an order relation on $X$ is defined by
	\begin{eqnarray}
	u\leq v &\Leftrightarrow& u-v\in X_+;\\
	u<v  &\Leftrightarrow&  u-v\in X_+ \hbox{ and } u\neq v.
	\end{eqnarray}
\end{definition}

\par Next, we need a result that allow us to compare mild solutions. 
\begin{theorem}
	\label{th-comparison-parabolic}
	Suppose that $f_1$ and $f_2$ satisfy hypothesis \textbf{(H)} and such that $f_1\leq f_2$. 
	For each $(s,u_0)\in \mathbb{R}\times X$, denote $u_1(t,s,u_0)$ and $u_2(t,s,u_0)$ 
	the mild solutions of the associated abstract Cauchy problems, respectively.
	Then,
	$u_1(t,s;u_0)\leq u_2(t,s;u_0)$ for every $t\geq s$, where both solutions are defined.
\end{theorem}
\par For the proof of Theorem \ref{th-comparison-parabolic} see Caraballo \textit{et al.} \cite[Theorem 3.1]{Caraballo-Langa-Obaya-Sanz}.

\medskip

\par Now we proceed ad the previous sections, first we consider a linear problem that we are going to compare Problem \eqref{th-abstract-Cauchy-problem}.
\par Let $a,b: \mathbb{R}\times\overline{\Omega}\to \mathbb{R}$ be continuous functions. We consider a non-homogeneous PDE
\begin{equation}\label{eq-applications-nonhomogenous-parabolic}
\left\{ 
\begin{array}{l l} 
u_t=\Delta u+ a(t,x) u +b(t,x),& \ t>s, \  x\in \Omega\\
u(s)=u_0,&  \ x\in \Omega\\
B(u)=0,& \ 
\end{array} 
\right.
\end{equation}
\par Define $g(t,x,u):=a(t,x)u+b(t,x)$. Then \eqref{eq-applications-nonhomogenous-parabolic} generates an evolution process 
$\mathcal{S}_g:=\{S_g(t,s): t\geq s\}$. 
 If $b=0$, then \eqref{eq-applications-nonhomogenous-parabolic} is associated with an
linear evolution process $\mathcal{S}_a:=\{S_a(t,s): t\geq s\}$.  
\par It is known that, if $a,b$ are locally Lipshitz on $t$, then
the evolution process $\mathcal{S}_g$ is given by the formula of variation of constants in terms of $\mathcal{S}_a$ and $b^e$,
see Henry \cite[Theorem 7.1.4]{Henry-1}.
This fact also holds true under weaker conditions over $a$ and $b$, as we show in the following lemma. The proof follows the same line of arguments of \cite[pp. 224-225]{Carvalho-Langa-Robison-book}.
\begin{lemma}\label{Lemma-constant-variation-mild-solutions}
	Suppose that $a,b$ are continuous. Then
	the evolution process $\mathcal{S}_g$ associate to \eqref{eq-applications-nonhomogenous-parabolic} is given by
	\begin{equation}\label{eq-nonautonomous-variation-constant-formula}
	S_g(t,s)=S_a(t,s)+\int_{s}^{t}S_a(t,r)b^e(r)dr, \ t\geq s.
	\end{equation}	
\end{lemma}
\begin{proof}
	\par Let $(s,u_0)\in \mathbb{R}\times X$ fixed and define 
	$u(t):=S_g(t,s)u_0$ the mild solution of \eqref{eq-applications-nonhomogenous-parabolic} and 
	\begin{equation*}
	v(t):=S_a(t,s)u_0+\int_{s}^{t}S_a(t,r)b^e(r)dr, \ \ t\geq s.
	\end{equation*}	
	We will show that $v(t)=u(t)$, for every $t\geq s$. 
	Indeed, note that
	\begin{equation*}
	\begin{split}
	v(t)-u(t)=&\int_{s}^{t}e^{A(t-r)}a^e(r)S_a(r,s)u_0dr+\int_{s}^tS_a(t,r)b^e(r)dr\\
	&-\int_{s}^{t}e^{A(t-r)}[a^e(r)u(r)+b^e(r)]dr\\
	=&  \int_{s}^{t}e^{A(t-r)}a^e(r)S_a(r,s)u_0dr-\int_{s}^{t}e^{A(t-r)}a^e(r)u(r)dr\\
	&+\int_{s}^{t}\int_{r}^{t}e^{A(t-\tau)}a^e(\tau)S_a(\tau,r)b^e(r) d\tau dr.
	\end{split}
	\end{equation*}
	Now, by Fubini's Theorem 
	\begin{equation*}
	\int_{s}^{t}\int_{r}^{t}e^{A(t-\tau)}a^e(\tau)S_a(\tau,r)b^e(r) d\tau dr
	=\int_{s}^{t}e^{A(t-\tau)}a^e(\tau)\int_{s}^{\tau}S_a(\tau,r)b^e(r) dr d\tau.
	\end{equation*}
	Therefore,
	\begin{equation*}
	v(t)-u(t)=\int_{s}^{t}e^{A(t-\tau)}a^e(\tau)(v(\tau)-u(\tau))d\tau.
	\end{equation*}
	By Gr\" ownwall's inequality, we conclude that $v(t)=u(t)$ for every $t\geq s$, and the proof is complete.
\end{proof}

\par Now, we are in position to apply the same approach provided for nonautonomous ODEs, in Section \ref{sec-app-odes}, to parabolic nonautonomous PDEs. 
\smallskip
In this Subsection we assume that:
\par \textbf{(H1)} $f(t,x,0)\geq 0$, for all $(t,x)$.
\par \textbf{(H2)} there exist $a,b:\mathbb{R}\times\overline{\Omega}\to \mathbb{R}$ continuous functions such that $f(t,x,u)\leq a(t,x)u+b(t,x)$, $u\geq 0$.

%
\smallskip
\par Conditions \textbf{(H1)} and \textbf{(H2)} over $f$ imply that
for each $(s,u_0)\in \mathbb{R}\times X$ the mild solution $u(t,s;u_0)$ of 
\eqref{th-abstract-Cauchy-problem} is defined for every $t\geq s$. 
Furthermore, by Theorem \ref{th-comparison-parabolic}, the associated evolution process $\mathcal{S}_f$ satisfies
\begin{equation}\label{eq-fundamental-comparison-parabalic}
	0\leq S_f(t,s)u_0\leq S_g(t,s)u_0, \  u_0\in X_+.
\end{equation}

\par Now, we study the asymptotic behavior of the evolution process $\mathcal{S}_f:=\{S_f(t,s):t\geq s\}$ over the metric space $(X_+,d)$, where $d(u,v)=\|u-v\|_X$, for $u,v\in X_+$. 
Finally, we write
$B^+[0,R]:=X_+\cap B[0,R]$, for any $R>0$.

\par From \textbf{(H1)} and \textbf{(H2)}, we have $b(t,x)\geq 0$, $(t,x)\in \mathbb{R}\times X$, so $b^e$ is a function that take values in $Y=C(\overline{\Omega})$, not necessarily in $C_0(\overline{\Omega})$. Thus, we define the family of spaces for non-homogeneous $b^e$ as
\begin{equation*}
C_\eta(\mathbb{R},Y):=\{b:J\to Y: \sup_{r\in \mathbb{J}}e^{-\eta|r|}\|b(r)\|_Y \}, \hbox{ for each } \eta\in \mathbb{R}.
\end{equation*}
\par For Neumann or Robin boundary conditions $X=\overline{D(A)}=Y$. Thus, to fix the notation, we use the space defined above for any boundary condition, Dirichlet, Neumann or Robin.

\begin{theorem}[Existence of Pullback
	$\mathcal{D}_\gamma$-attractor]
	\label{th-existence-pullback-attractors-line-parabolic}
	Let $f:\mathbb{R}\times \overline{\Omega}\times \mathbb{R}\to \mathbb{R}$ satisfies 
	\textbf{(H), (H1)} and \textbf{(H2)} and 
	$\mathcal{S}_f$ be the evolution process induced by 
	\eqref{th-abstract-Cauchy-problem} and $\mathcal{D}_\gamma$ the universe defined on 
	Example \ref{def-example-universe}. 
	\par Assume that $\mathcal{S}_a$ admits a NEDII on $\mathbb{R}^-$ with $X^s_{II}(\alpha,\delta)$ and $\Pi^u=0$ and NEDII on $\mathbb{R}^+$ with $X^s_{II}(\beta,\nu)$ and $\Pi^u_+(t)=0=\Pi^u_-(s)$, $t\geq 0\geq s$.
	\par Suppose that $\lambda<\alpha/\delta$, then
	for 
	$b:\mathbb{R}\times \overline{\Omega}\to \mathbb{R}$ such that
	$b^e_-:=b^e|_{\mathbb{R}^-}\in C_{\lambda\delta}(\mathbb{R}^-,Y)$
	and
	$b^e_+:=b^e|_{\mathbb{R}^+}\in C_{\eta\nu}(\mathbb{R}^+,Y)$,
	there is a pullback $\mathcal{D}_\gamma$-attractor 
	$\widehat{\mathcal{A}}=\{\mathcal{A}_\gamma(t):t\in \mathbb{R}\}$
	for
	$\mathcal{S}_f$, for every $\gamma\in [0,\alpha)$.
	\par Moreover,
	$\mathcal{A}_\gamma(t)\subset B^+[0,R(t)]$, 
	for each $t\in \mathbb{R}$, for some $R:\mathbb{R}\to \mathbb{R}$ satisfying
	$R|_{\mathbb{R}^-}\in C_{(1+\lambda)\delta}(\mathbb{R}^-)$ and
	\begin{itemize}
		\item If $\eta>-\beta/\nu$, then $R|_{\mathbb{R}^-_*}\in C_{(1+\eta)\nu}(\mathbb{R}^-_*)$.
		\item If $\eta<-\beta/\nu$ then $R|_{\mathbb{R}^-_*}\in C_{(\nu-\beta)}(\mathbb{R}^-_*)$.
		\item If $\eta=-\beta\nu$, then 
		$R|_{\mathbb{R}^-_*}\in C_{(\nu-\beta)+\epsilon}(\mathbb{R}^-_*)$, for every $\epsilon>0$.
	\end{itemize} 
\end{theorem}

\begin{proof}
	Let 
	$b:\mathbb{R}\times \Omega\to \mathbb{R}$ be a continuous function satisfying hypotheses above, 
	$\widehat{D}=\{D(t): t\in \mathbb{R}\}\in \mathcal{D}_\gamma$, and $u_s\in D(s)$, for each $\gamma\in [0,\alpha)$ and $s\leq t\leq 0$ fixed. 
	From \eqref{eq-fundamental-comparison-parabalic} and Lemma
	\ref{Lemma-constant-variation-mild-solutions},
	\begin{eqnarray*}
		\|S_f(t,s)u_s\|_X&\leq& \|S_g(t,s)u_s\|_X\\
		&=& \|S_a(t,s)u_s\|_X+\int_{s}^{t}\|S_a(t,r)b^e(r)\|_Xdr.
	\end{eqnarray*}
	At this point, to deal with the integral above, we split the argument in two cases.
	\par \textbf{Case I:} If $X=Y$, then 
	\begin{equation}\label{eq-caseI}
	\|S_a(t,s)b^e(s)\|_X\leq \|S_a(t,s)\|_{\mathcal{L}(X)}\|b^e(s)\|_X, \ \ t\geq s.
	\end{equation}
	\par \textbf{Case II:} If $X=C_0(\overline{\Omega})$. Note that,
	there exists $C>0$ such that
	\begin{equation}\label{eq-boundedness-between-X-and-Y}
	\|S_a(t,r)b^e(r)\|_X\leq CMe^{\delta|t|}e^{-\alpha(t-r)}\|b^e(r)\|_Y, \ \ t\geq r,
	\end{equation}
	Indeed, for $v\in Y$ and $\epsilon>0$ small,
	\begin{equation*}
	\begin{split}
	\|S_a(s+\epsilon,s)v\|_{Y}\leq &\, \|e^{A \epsilon} v\|_{Y} +
	\int_{s}^{s+\epsilon} \|e^{A(s+\epsilon-r)}a^e(r) S_a(r,s)v\|_{Y} dr\\
	\leq &\, Ce^{\omega\epsilon}\|v\|_Y+ Ce^{\omega\epsilon}\sup_{s\leq r\leq s+1}\{\|a^e(r)\|_Y\} \int_{s}^{s+\epsilon}\|S_a(r,s)v\|_{Y} dr,
	\end{split}
	\end{equation*}
	where in the last inequality was used that 
	$\{e^{At}: t\geq 0\}$ is a semigroup on $Y$, and hence, there are $C,\omega>0$ such that $\|e^{At}\|_{\mathcal{L}(Y)}\leq Ce^{\omega t}$, $t\geq0$.
	Then 
	\begin{equation*}
	\limsup_{\epsilon\to 0}\|\|S(r+\epsilon,r)b(r)\|_X\leq C \|b(r)\|_Y.
	\end{equation*}
	\par Thus, \eqref{eq-boundedness-between-X-and-Y} will follows from
	\begin{equation}
	\|S(t,r)b(r)\|_X\leq \limsup_{\epsilon\to 0} \|\|S(t,r+\epsilon)\|_{\mathcal{L}(X)}
	\limsup_{\epsilon\to 0}\|\|S(r+\epsilon,r)b(r)\|_X.
	\end{equation}
	\par Now, following similar ideas of the proof of Theorem \ref{th-existence-pullback-attractors-line}, 
	we show that exists $R$ satisfying the properties of the statement such that
	the family of bounded sets
	\begin{equation*}
	\widehat{B}:=\{B^+[0,R(t)]: t\in \mathbb{R}\}, 
	\end{equation*}
	pullback absorbs $\widehat{D}$, i.e., for each $t$ fixed, there exist $s_0=s_0(t,\widehat{D})\leq t$ such that 
	\begin{equation*}
	S_f(t,s)D(s) \subset B^+[0,R(t)], \hbox{ for every } s \geq s_0.
	\end{equation*}
	\par Since $\{S_f(t,s):t>s\}$ is a family of compact operators, the sets
	$K(t):=S(t,t-1)B^+[0,R(t-1)]$, define a family of precompact sets 
	$\widehat{K}:=\{K(t): t\in \mathbb{R}\}$ that pullback absorbs every element of $\mathcal{D}_\gamma$, for each $\gamma\in (0,\alpha)$.
	\par Therefore, by Theorem \ref{th-characterization-existence-pullback-attractors},
	there exists a pullback $\mathcal{D}_\gamma$-attractor $\widehat{\mathcal{A}}_\gamma$
	for the evolution process $\mathcal{S}_f$. Moreover, by the minimality property,
	$\widehat{\mathcal{A}}_\gamma\subset \widehat{B}$, and the proof is complete.
\end{proof}

\begin{remark}
	Every result of Subsection \ref{subssec-comparison-systems-odes} has a corresponding version for parabolic PDEs. 
\end{remark}

\par As a consequence of Theorem \ref{th-existence-pullback-attractors-line-parabolic}, we state the following result on pullback and forward attraction.

\begin{theorem}[Complete Asymptotic Analysis]
	\label{th-existence-pullback-attractors-line-verion-2-parabolic}
Let $f:\mathbb{R}\times \overline{\Omega}\times \mathbb{R}\to \mathbb{R}$ satisfies 
\textbf{(H), (H1)} and \textbf{(H2)} and 
$\mathcal{S}_f$ be the evolution process induced by 
\eqref{th-abstract-Cauchy-problem} and $\mathcal{D}_\gamma$ the universe defined on 
Example \ref{def-example-universe}.  
\par Suppose that $\mathcal{S}_a$ admits a NEDII on $\mathbb{R}$ with $X^s(\alpha,\delta)$ and $\Pi^u=0$. 
\par For $\lambda\in(-\alpha/\delta,\alpha/\delta)$ and $b\in C_{\lambda \delta}(\mathbb{R})$,
then, for every $\gamma\in [0,\alpha)$,
there exists a pullback $\mathcal{D}_\gamma$-attractor 
$\widehat{\mathcal{A}}_\gamma
=\{{\mathcal{A}}_\gamma(t): t\in \mathbb{R\}}$ for $\mathcal{S}_f$ such that
$\mathcal{A}_\gamma(t)\subset B^+[0,R(t)]$ for all $t\in \mathbb{R}$, for some
$R\in C_{(1+\lambda)\delta}(\mathbb{R})$. 
\par Additionally, if
$\alpha>\delta$ and $\lambda\in (-\alpha/\delta,-1]$, then
	\begin{itemize}
		\item  There exists a closed, bounded subset $B$ of $X$ which forward absorbs every bounded subset of $X$.
		\item There exists a family of compact subsets $\{K(t): t\in \mathbb{R}\}$ of $X$ that forward absorbs bounded subsets of $X$ and $\cup_{t\in \mathbb{R}}K(t)\subset B$.
		\item The pullback $\mathcal{D}_\gamma$-attractor $\widehat{\mathcal{A}}_\gamma$ is uniformly bounded, for every $\gamma\in [0,\alpha)$.
	\end{itemize}
\end{theorem}

\begin{example}\label{example-simpler-example-parabolic}
	Define $a(t):=-c-dt\sin(t)I_X$ where $c,d>0$ such that 
	$A_0:=A-c$ has spectrum satisfying $\sigma(A_0)\subset (-\infty,-r)$, for some $r>d$. Then $\mathcal{S}_a$ satisfies hypothesis \textbf{(H3)}.
\end{example}

%

\subsection{Existence of nonuniform exponential dichotomies}
\label{subsec-existence-NEDII}
In this subsection our goal is to provide examples evolution processes that admits NEDII.
We propose a theory that relates nonuniform exponential dichotomies with continuous separation for \textit{strongly monotone dynamical systems}. From these analysis, we provide examples of evolution processes that admits NEDII, using the examples presented in Subsection \ref{subsec-examples_NEDII}.
\par To obtain a continuous separation we have to consider additional properties for monotone dynamical systems:

\begin{definition}
	Let $(X,\leq)$ a ordered Banach space with a nonnegative cone $X_+$. 
	We say that 
	$X$ is a \textbf{strongly ordered Banach space} if the interior of the cone
	$Int X_+$ is nonempty and there is a stronger order relation on $X$ is defined by
		\begin{equation}
	u\ll v \Leftrightarrow u-v\in  \hbox{Int}\,X_+.
	\end{equation}
\end{definition}
\par Since for Neumann or Robin boundary conditions $X=C(\overline{\Omega})$, the cone of nonnegative vectors is defined by
$X_+:=\{u\in X: u(x)\geq 0,  \ \forall \,x\in \overline{\Omega}\}$ and its interior by
$\hbox{Int}\,X_+:=\{u\in X: u(x)> 0,  \ \forall \,x\in \overline{\Omega}\}$. 
\par However, for Dirichlet boundary conditions, since the choice $X$ as $C_0(\overline{\Omega})$
has a positive cone with an empty interior, we will have to resort to an intermediate space. Hence we will treat each case of boundary condition separately.

\par  First, let us consider a linear scalar parabolic PDE with Neumann or Robin boundary conditions
\begin{equation}
\label{eq-parabolic-strogly-positive}
\left\{ 
\begin{array}{l l} 
u_t=\Delta u+ h(t,x) u,& \ t>s, \  x\in \Omega,\\
u(s)=u_0,&  \ x\in \Omega,\\
B(u)=0,& \ x\in \partial \Omega,
\end{array} 
\right.
\end{equation}
where $h(t,x)$ continuous, bounded and uniformly continuous in $t$. 
The hull of $h$,
$\mathcal{P}:=Hull(h)$ is the closure for the compact-open topology of the set of 
$t$-translates of $h$, $\{h(t+\cdot,\cdot): t\in \mathbb{R}\}$.
On the compact metric space $\mathcal{P}$ we define the translation map 
$\mathbb{R}\times \mathcal{P}\to \mathcal{P}$, $(t,p)\mapsto p\cdot t$ given by 
$p\cdot t (s,x)=p(s+t,x)$ ($s\in \mathbb{R} \hbox{ and } x\in \overline{\Omega}$) is
a continuous flow over $\mathcal{P}$. 
We denote by $\tau:\mathbb{R}^+\times \mathcal{P}\times X\to \mathcal{P}\times X$ the linear skew product semi-flow 
\begin{equation*}
\tau(t)(p,z):=(p\cdot t, \phi(t,p)z),
\end{equation*}
induced by the mild solutions $\phi(t,p)z:=S_p(t,0)z$ 
of the abstract Cauchy problem associated to 
\eqref{eq-parabolic-strogly-positive}. 
In particular, $\phi(t,p)$ are bounded linear operators on $X$ which are compact for $t>0$
and satisfy the co-cycle property $\phi(t+s,p)=\phi(t,p\cdot s)\phi(s,p)$, $t,s\geq 0$, $p\in \mathcal{P}$. 
Moreover, these operators $\phi(t,p)$ are \textbf{strongly 
positive} in $X$, i.e., for $p\in \mathcal{P}$ 
and $t>0$, $\phi(t,p)z\gg0$ if $z>0$. 
\par Therefore, 
$\tau$ admits a \textbf{continuous separation} (see Poláčik and Tereščák \cite{Polacik-Terescak} in the discrete case and Shen and Yi \cite{Shen-Yi} in the continuous case).
This means that there are two families of subspaces
$\{X_1(p)\}_{p\in \mathcal{P}}$ and $\{X_2(p)\}_{p\in \mathcal{P}}$ of $X$ which satisfy:
\begin{enumerate}
	\item $X=X_1(p)\oplus X_2(p)$;
	\item $X_1(p)=\langle e(p) \rangle$, with $e(p)\gg0$ and
	$\|e(p)\|_X=1$ for any $p\in \mathcal{P}$;
	\item $X_2(p)\cap X_+={0}$, for any $p\in \mathcal{P}$;
	\item for any $t>0$, $p\in \mathcal{P}$,
	\begin{eqnarray*}
		&\phi(t,p)X_1(p)=X_1(p\cdot t),&\\
		&\phi(t,p)X_2(p)\subset X_2(p\cdot t);& 
	\end{eqnarray*}
	\item there are $M,\nu>0$ such that for any $p\in P$, $z\in X_2(p)$ with $\|z\|=1$ and 
	$t>0$ we have
	\begin{equation}\label{eq-property-5-continuous-separation}
	\|\phi(t,p)z\|_X\leq Me^{-\nu t}\|\phi(t,p)e(p)\|_X. 
	\end{equation}
\end{enumerate}
In this situation, the $1$-dim invariant subbundle
\begin{equation*}
\bigcup_{p\in P} \{p\}\times X_1(p)
\end{equation*}
is called the \textbf{principal bundle}.
\par Let $Q_i(p):X\to X_i(p)$ be the projection of $X$ onto $X_i(p)$, for each $p\in \mathcal{P}$
and $i=1,2$. Since the mapping $p\mapsto Q_i(p)$ is continuous
and $\mathcal{P}$ is compact, the projections are uniformly bounded over $\mathcal{P}$, i.e., there exist 
$C_1,C_2>0$ such that
\begin{equation}\label{eq-projections-continuous-separation-uniformlybounded}
\sup_{p\in P}\|Q_i(p)\|_{\mathcal{L}(X)}\leq C_i, \ i=1,2.
\end{equation} 

\begin{remark}\label{remark-principle-bundle}
\par The continuous separation allows to associate to $\phi$ a $1$-dim 
continuous linear co-cycle $c(t,p)>0$, given by
\begin{equation*}
\phi(t,p)e(p)=c(t,p)e(p\cdot t), \ t\geq 0, \ p\in P.
\end{equation*}
For each $p\in \mathcal{P}$ and $t\geq s$, we set
\begin{equation*}
S_p(t,s):=\phi(t-s,p\cdot s),
\end{equation*}
then
$\mathcal{S}_p:=\{S_p(t,s): t\geq s\}$ is an evolution process associated with
$\phi$, and similarly, we consider a scalar evolution process associated with the scalar co-cycle $c$
\begin{equation*}
\mathfrak{s}_p(t,s):=c(t-s,p\cdot s).
\end{equation*}
Therefore,
\begin{equation}\label{eq-evolutionprocess-principalbunddle}
S_p(t,s)e(p\cdot s)=\mathfrak{s}_p(t,s)e(p\cdot t), \ t\geq s, \ p\in \mathcal{P}.
\end{equation}
\par Since $(p\cdot s)^e(t-s)=p^e(t)$, $t\geq s$, 
this notation is coherent with this section: 
given $p\in \mathcal{P}$,
for each $(s,u_0)\in \mathbb{R}\times X$,
the mapping $t\mapsto S_p(t,s)u_0$ is the mild solution of 
the abstract Cauchy problem
\begin{equation*}
\left\{ 
\begin{array}{l l} 
u_t=A u+ p^e(t) u,& \ t>s,\\
u(s)=u_0.& 
\end{array} 
\right.
\end{equation*}
For $p=h$ we obtain the abstract Cauchy problem induced by 
\eqref{eq-parabolic-strogly-positive}.
\end{remark}

\par Now, we are read to characterized existence of nonuniform
exponential dichotomies for these dynamical systems that admits a continuous separation.

\begin{theorem}\label{th-1-dim-to-infinite-dim--existence-NEDII}
	Let $p\in \mathcal{P}$, and $\mathcal{S}_p$ and $\{\mathfrak{s}_p(t,s): t\geq s\}$ the evolution processes defined above. 
	Then, $\mathcal{S}_p$ admits a NEDII with $\Pi^u_\mathcal{S}=0\in \mathcal{L}(X)$ if and only if $\{\mathfrak{s}_p(t,s): t\geq s\}$ admits a NEDII with $\Pi^u_\mathfrak{s}=0\in \mathcal{L}(\mathbb{R})$.
\end{theorem}
\begin{proof}
	\par First, assume that $\mathcal{S}_p$ admits a NEDII with $\Pi^s_\mathcal{S}=Id_X$, 
	there exist $\widetilde{M},\alpha>0$, and 
	$\delta\geq 0$ such that
	\begin{equation*}
		\|S_p(t,s)\|_{\mathcal{L}(X)}\leq \widetilde{M}e^{\delta|t|}e^{-\alpha(t-s)}, \ t\geq s.
	\end{equation*}
	To conclude that $\mathfrak{s}_p$ admits a NEDII with $\Pi^s_\mathfrak{s}=Id_\mathbb{R}$,
	note that
	\begin{eqnarray}
	\mathfrak{s}_p(t,s)=\|S_p(t,s)e_1(p\cdot s)\|_X, \ t\geq s.
	\end{eqnarray}
	\par Reciprocally, suppose there are $\widetilde{M},\alpha>0$, and 
	$\delta\geq 0$ such that
	\begin{equation*}
		\mathfrak{s}_p(t,s)\leq \widetilde{M}e^{\delta|t|}e^{-\alpha(t-s)}, \ t\geq s.
	\end{equation*}
	Let $x\in X$, then there exists $x_i(p\cdot s)\in X_i(p\cdot s)$, for $i=1,2$, such that
	$u=u_1(p\cdot s)+u_2(p\cdot s)$. 
	Suppose without loss of generality that
	$u_2(p\cdot s)\neq 0$. 
	Thus, from \eqref{eq-property-5-continuous-separation}, 
	\eqref{eq-projections-continuous-separation-uniformlybounded} and \eqref{eq-evolutionprocess-principalbunddle}, we have that 
	\begin{eqnarray*}
		\|S_p(t,s)u\|_X&\leq& \|S_p(t,s)u_1(p\cdot s)\|_X+\|S_p(t,s)u_2(p\cdot s)\|_X\\
		&\leq& \mathfrak{s}_p(t,s)\|u_1(p\cdot s)\|_X+ Me^{-\nu(t-s)} \mathfrak{s}_p(t,s)\|u_2(p\cdot s)\|_X\\
		&\leq& (C_1+MC_2e^{-\nu(t-s)}) \mathfrak{s}_p(t,s)\|u\|_X\\
		&\leq & C \widetilde{M} e^{\delta|t|}e^{-\alpha(t-s)}\|u\|_X, \ t\geq s,
	\end{eqnarray*}
	where $C=C_1+MC_2$. The proof is complete.
\end{proof}

%

\par Now, we consider the case of Dirichlet boundary condition, i.e., Problem 
\ref{eq-parabolic-strogly-positive} with $Bu=u$. 
To obtain a continuous separation, we introduced the strongly ordered Banach space $E^\alpha$ defined as follows. 
\par Let us consider the realization of the Laplacian operator on $E:=L^p(\Omega)$ for a fixed 
$p\in (N,+\infty)$, i.e., the operator 
$A_p:D(A_p)\subset L^p(\Omega)\to L^p(\Omega)$ defined by 
$A_pu=\Delta u$ (in the weak sense) for $u\in D(A_p)$. Then $-A_p$ is sectorial, densely defined and $0\in \rho (A_p)$. 
Then the fractional power $E^\alpha=(D(-A_p^\alpha),\|\cdot\|_\alpha)$ is a Banach space with the norm $\|u\|_\alpha=\|(-A_p)^\alpha u\|_p$ and satisfies 
$E^\alpha \hookrightarrow C^1(\overline{\Omega})\cap C_0(\overline{\Omega})$ for $\alpha\in (1/2+N/(2p), 1)$, see Henry \cite[Theorem 1.6.1]{Henry-1}.
Then the partial strong order in the Banach space $E^\alpha$ with the cone of positive elements $E^\alpha_+=\{u\in E^\alpha: u(x)>0, \hbox{ for } x\in \Omega \}$, which has a nonempty interior
\begin{equation*}
	Int \, E^\alpha_+=\{u\in E^\alpha_+: u(x)>0, \hbox{ for } x\in \Omega \hbox{ and } \frac{\partial u}{\partial n}(x)=0, \hbox{ for } x\in \partial \Omega \}.
\end{equation*}
Furthermore, $\{e^{A_pt} t\geq 0\}$ is a strongly continuous analytic semigroup over 
$E^\alpha$, and $e^{A_pt}:E\to E^\alpha$ is a compact operator, for each $t>0$.
\par Recall that $\mathcal{P}$ is the hull of $h$ and consider the linear skew product semiflow $\tau:\mathbb{R}^+\times \mathcal{P}\times E^\alpha\to \mathcal{P}\times E^\alpha$  
\begin{equation*}
	\tau(t)(p,z):=(p\cdot t, \phi(t,p)z), \ (p,z)\in \mathcal{P}\times E^\alpha
\end{equation*}
induced by the mild solutions $\phi(t,p)z:=S_p(t,0)z\in E^\alpha$ 
of the abstract Cauchy problem
\begin{equation}
\label{eq-parabolic-strogly-positive-Dirichlet-abstract}
\left\{ 
\begin{array}{l l} 
u_t=Au+ p^e(t) u, & t>s\\
u(s)=u_0\in E^\alpha,&\end{array} 
\right.
\end{equation}
As in the case of Neumann/Robin boundary conditions, $\varphi(t,p)$ is strongly positive over $E^\alpha$, for each $(t,p)\in \mathbb{R}\times E^\alpha$. Hence, $\tau$ admits a continuous separation over $\mathcal{P}\times E^\alpha$. Hence, the same analysis for Neumann and Robin boundary conditions in the phase space $X$ holds true for Dirichlet boundary conditions in $E^\alpha$. 

\par Now we provide conditions that guarantees that the existence of nonuniform exponential dichotomies in $E^\alpha$ actually implies in $X$.

\begin{lemma}\label{lemma-NEDII-E^a-to-X}
	Let $a:\mathbb{R}\times \overline{\Omega}\to \mathbb{R}$ be a bounded continuous real valued function. Assume that
	the evolution process $\mathcal{S}_a=\{S_a(t,s): t\geq s\}$ associated with 
	\eqref{eq-applications-nonhomogenous-parabolic} with $b\equiv 0$
	admits NEDII on 
	$E^\alpha$ with $\Pi^u=0$. Then $\mathcal{S}_a$ admits NEDII on $X$  with $\Pi^u=0$.
\end{lemma}
\begin{proof}
	Assume that $\mathcal{S}_a$ admits a NEDII on $E^\alpha$ with $\Pi^u=0$, i.e.,
	there exist $M,\delta,\alpha>0$ such that
	\begin{equation}
	\|S_a(t,s)\|_{\mathcal{L}(E^\alpha)}\leq Me^{\delta|t|} e^{-\alpha(t-s)}, \ \ t\geq s.
	\end{equation}
	Let $u\in X$ with $\|u\|_X=1$. 
	Since $E^\alpha \hookrightarrow X$, there exists $C>0$ such that
	\begin{equation}\label{eq-lemma-NEDII-E^a-to-X-1}
	\|S_a(t,s)u\|_X\leq C \|S_a(t,s)u\|_{E^\alpha}, \ \ t\geq s.
	\end{equation}
	Now, suppose that $t\geq s+1$, we have that $S_a(s+1,s)u\in E^\alpha$, and 
	\begin{equation}\label{eq-lemma-NEDII-E^a-to-X-2}
	\|S_a(t,s)u\|_{E^\alpha}\leq  Me^{\delta|t|} e^{-\alpha(t-(s+1))} \|S_a(s+1,s)u\|_{E^\alpha}, \ \ t\geq s.
	\end{equation}
	Note that, 
	\begin{equation}
	\|S_a(s+1,s)u\|_{E^\alpha}\leq \|e^{A 1} u\|_{E^\alpha} +
	\int_{s}^{s+1} \|e^{A(s+1-r)}a^e(r) S_a(r,s)u\|_{E^\alpha} dr
	\end{equation}
	First, since $e^{A 1}: X\to E^\alpha$ is a compact operator and
	$\|u\|_X=1$, there exists 
	$C>0$ such that $\|e^{A 1} u\|_{E^\alpha}\leq C$. 
	Second, the fact that $a$ is bounded
	$\sup_{r\in \mathbb{R}} \|a(r)\|_E\leq C_1$, for some $C_1>0$, and
	\begin{equation}
	\|S_a(s+1,s)u\|_{E^\alpha}\leq C +
	C_1 \int_{s}^{s+1} (s+1-r)^{-\alpha}\|S_a(r,s)u\|_{E^\alpha} dr.
	\end{equation}
	Form a Singular Gronwall's Lemma \cite[Lemma 6.24]{Henry-1}, we obtain that 
	$\|S_a(s+1,s)u\|_{E^\alpha}\leq C_3$, for some $C_3=C_3(C,C_1,\alpha)>0$.
	Thus, from \eqref{eq-lemma-NEDII-E^a-to-X-1} and \eqref{eq-lemma-NEDII-E^a-to-X-2}
	we conclude that $\mathcal{S}_a$ admits NEDII on $X$ with $\Pi^u=0$, and the proof is complete.
\end{proof}

\par Now, we are ready to state the main result of this subsection.
The proof use similar ideas presented in \cite[page 18]{Caraballo-Langa-Obaya-Sanz}.
\begin{theorem}\label{th-exsitence-NEDII-parabolic}
		Let $g:\mathbb{R}\to \mathbb{R}$ be such that the scalar ODE
		\begin{equation}\label{eq-scalar-to-parabolic-NEDII}
		\dot{x}=g(t)x, \ x(s)=x_0.
		\end{equation}
		is associated with an evolution process
	$\mathcal{T}=\{T(t,s): t\geq s\}$  over $\mathbb{R}$
	that admits a NEDII with $\Pi^u_\mathcal{T}=0 \in \mathcal{L}(\mathbb{R})$. 
	\par Then, there exists a function 
	$a:\mathbb{R}\times \overline{\Omega}\to \mathbb{R}$ such that 
	the abstract Cauchy problem 
	\begin{equation*}
	\left\{ 
	\begin{array}{l l} 
	u_t=A u+ a^e(t) u,& \ t>s,\\
	u(s)=u_0,& 
	\end{array} 
	\right.
	\end{equation*}
	induces an evolution process 
	 $\mathcal{S}_a=\{S_a(t,s): t\geq s\}$ 
	 that admits a NEDII with $\Pi^u=0$ over $X$ in the case of Neumann or Robin boundary conditions and over $E^\alpha$ in the case of Dirichlet boundary condition.
	 \par Additionally, if $g$ is bounded, then
	 $\mathcal{S}_a=\{S_a(t,s): t\geq s\}$ admits a NEDII in $X$, for any type of boundary condition.
	\end{theorem}

\begin{proof} 
	From Johnson \textit{et al.} \cite[Lemma 3.2]{Johnson-Palmer-Sell-85}, we know that there exists
	a smooth cocycle $c_0(t,p)=exp\int_{0}^{t}d(p\cdot r) \,dr$ for some
	continuous map $d:\mathcal{P}\to \mathbb{R}$, and a continuous map $f:\mathcal{P}\to \mathbb{R}\setminus \{0\}$ such that 
	\begin{equation}\label{eq-cohomology-c_0-c}
	c(t,p)=f(p\cdot t) c_0(t,p) f(p)^{-1}, \ t\in \mathbb{R}, \ p\in \mathcal{P}.
	\end{equation}
	Consider $k_p(t):=g(t)-d(p\cdot t)$ and
	define
	\begin{equation*}
	\tilde{\mathfrak{s}}_0(t,s;p):=c_0(t-s,p \cdot s)e^{\int_s^tk_p(r)dr}, \ t\geq s, \ p\in \mathcal{P}.
	\end{equation*} 
	Then, for each $(s,x_0)\in \mathbb{R}^2$, the mapping $t\mapsto \tilde{\mathfrak{s}}_0(t,s;p)x_0$ is the solution of \eqref{eq-scalar-to-parabolic-NEDII} and
	$T(t,s)=\tilde{\mathfrak{s}}_0(t,s;p)$, $t\geq s$ and $p\in \mathcal{P}$. 
	Since $\mathcal{T}$ admits NEDII with $\Pi^u_\mathcal{T}=0 \in \mathcal{L}(\mathbb{R})$ and 
	$f$ is bounded,
	from \eqref{eq-cohomology-c_0-c} we obtain that 
	the scalar evolution process $\{\tilde{\mathfrak{s}}(t,s;p):t\geq s\}$ defined by
	\begin{equation}
	\tilde{\mathfrak{s}}(t,s;p):=\mathfrak{s}(t,s;p) e^{\int_s^tk_p(r)dr},
	\end{equation}
	admits a NEDII with $\Pi^u_\mathcal{T}=0 \in \mathcal{L}(\mathbb{R})$, where $\mathfrak{s}(t,s;p) $ is defined in Remark \ref{remark-principle-bundle}.
	 Then, by similar arguments to those used in the proof of Theorem \ref{th-1-dim-to-infinite-dim--existence-NEDII}, we obtain that the evolution process defined by
	\begin{equation*}
	\widetilde{S}(t,s;p):=S_p(t,s)e^{\int_s^tk_p(r)dr}, \ t\geq s, \ p\in \mathcal{P},
	\end{equation*} 
	admits NEDII with $\Pi^u=0$ in $X$ in the case of Neumann or Robind boundary conditions or in $E^\alpha$ for Dirichlet boundary conditions.
	\par Let $a(t,x):=h(t,x)+k_h(t)$, then $S_a(t,s)=\widetilde{S}(t,s;h)$, for $t\geq s$.
	\par Finally, if $g$ is bounded then $a:\mathbb{R}\times \overline{\Omega}\to \mathbb{R}$ is bounded and the proof follows from Lemma \ref{lemma-NEDII-E^a-to-X}.
	\end{proof}

	Thanks to Theorem \ref{th-exsitence-NEDII-parabolic}, each example from Subsection \ref{subsec-examples_NEDII} induces a linear evolution process over $X$ that admits NEDII. For instance:

	\begin{corollary}
		There exists $a:\mathbb{R}\times \Omega\to \mathbb{R}$ such that the associated evolution process $\mathcal{S}_a$ over $X$ admits a NEDII on $\mathbb{R}$, and does not admit any NEDI over $X$. 
	\end{corollary}
	\begin{proof}
		Let $f$ be the function satisfying conditions of Proposition \ref{Example-f-continuous-bounded} and apply Theorem \ref{th-exsitence-NEDII-parabolic}.
		\end{proof}
	
\begin{remark}
	Every result of this subsection holds true for any nonuniform exponential dichotomy of type I and II over any interval $\mathbb{R},\mathbb{R}^+,\mathbb{R}^-$.
\end{remark}	
%
%
%
%

\subsection{The adjoint problem}
\label{subsec-adjoint-parbolic-pde}
\par We consider the adjoint problem of a parabolic
and we established a fundamental relation between NEDI and NEDII similar to the invertible case Subsection \ref{subsec-NED-invertile-ep}.
\par Now, we consider the adjoint problem of \eqref{eq-applications-nonhomogenous-parabolic} with $b\equiv 0$, i.e., the backward parabolic equation
\begin{equation*}
\left\{ 
\begin{array}{l l} 
u_t=-\Delta u- a(t,x) u,& \ t<s, \  x\in \Omega\\
u(s)=u_0,& \ x\in \Omega\\
B^*(u)=0,& \ x\in \partial \Omega 
\end{array} 
\right.
\end{equation*}
\par Denote $S^*(t,s)$ to be the weak solution operator defined for each $t<s$, for details we recommend see \cite{Henry-1,Mierczynski-Shen-02}.

\par The proof of the following proposition can be found in 
\cite[Proposition 2.3.3]{Mierczynski-Shen-02}.
\begin{proposition}
	The adjoint operator of $S(t,s)$ is given by
	\begin{equation*}
	[S(t,s)]^*=S^*(s,t), \hbox{ for any } s<t.
	\end{equation*}
\end{proposition}	
\par The \textbf{dual evolution operator} of the evolution process $\mathcal{S}$ is define as
\begin{equation*}
\widetilde{S}(t,s):=S^*(-t,-s), \ \hbox{ for every } t\geq s.
\end{equation*}
\par Thus $\widetilde{\mathcal{S}}:=\{\widetilde{S}(t,s):\, t\geq s\}$ defines the dual evolution process correspondent to $\mathcal{S}$, and it is associated with the solutions of
\begin{equation*}
\left\{ 
\begin{array}{l l} 
u_t=\Delta u+ a(-t,x) u,& \ t>s, \  x\in \Omega\\
B(u)=0,& \ x\in \partial \Omega 
\end{array} 
\right.
\end{equation*}
\par Similarly as Theorem \ref{th-fundamental-relation-between-type1-type2} we have the following result for the adjoint solution operator.
\begin{theorem}\label{th-fundamental-relation-between-type1-type2-parabolic}
	 Suppose that $\mathcal{S}$ admits a NEDI (NEDII) with bound $M(t)=Me^{\upsilon|t|}$, for $t\in \mathbb{J}$, exponent $\omega>0$, and family of projections $\Pi^u$, for some $M,\upsilon>0$. Then the dual evolution process $\widetilde{\mathcal{S}}$ admits a NEDII (NEDI) with bound $M(t)$ and exponent $\omega>0$, and family of projections
	\begin{equation*}
			\widetilde{\Pi}^u
				:=\{[\Pi^{u}(-t)]^*: t\in \mathbb{R}\}.
			\end{equation*}
\end{theorem}
\begin{proof}
	Assume that $\mathcal{S}$ admits a NEDI, we will show that the dual evolution process $\widetilde{\mathcal{S}}$ admits NEDII. 
	Let $\{\Pi^u(t): t\in\mathbb{R}\}$ be the family of projections that determines a NEDII for $\mathcal{S}$.
	Define $\widetilde{\Pi}^u(t)=\Pi^u(-t)^*$ and $\widetilde{\Pi}^s(t)=Id_X-\widetilde{\Pi}^u(t)$, $t\in \mathbb{R}$.
	It is straightforward to verify that
	\begin{equation*}
	\widetilde{S}(t,s)\widetilde{\Pi}^u(s)=\widetilde{\Pi}^u(t)\widetilde{S}(t,s), \  t\geq s.
	\end{equation*}
	Note that
	\begin{equation}\label{eq-fundamental-relation-betweeen-type1-type2-parabolic-1}
	\|\tilde{S}(t,s)\widetilde{\Pi}^s(s)\|_{\mathcal{L}(X)}	=\|[S(-s,-t)\Pi^s(-t)]^*\|_{\mathcal{L}(X^*)}\leq 
	De^{\nu |t|} e^{-\alpha(t-s)}, \ t\geq s.
	\end{equation}
	It is also straightforward to verify that
	\begin{equation*}
	\widetilde{S}(t,s)\widetilde{\Pi}^u(s)=\widetilde{\Pi}^u(t)\widetilde{S}(t,s), \  t\geq s.
	\end{equation*}
	\par Now, we prove that 
	$\tilde{S}(t,s):R(\widetilde{\Pi}^u(s)) \to R(\widetilde{\Pi}^u(t))$ is an isomorphism, for every $t\geq s$. 
	Indeed, since $S(-s,-t):R(\Pi^u(-t))\to R(\Pi^u(-s))$ is an isomorphism for every $t\geq s$, there exists 
	$S(-t,-s)\in \mathcal{L}(R(\Pi^u(-s)),R(\Pi^u(-t))$ the inverse of $S(-s,-t)$. 
	Hence
	\begin{equation}\label{eq-inverse-adjoint-projection}
	\widetilde{S}(s,t):=[S(-t,-s)]^*\in \mathcal{L}([R(\Pi^u(-t))]^*,[R(\Pi^u(-s))]^*).
	\end{equation} 
	Moreover, $\widetilde{S}(s,t)$ is the inverse of $\tilde{S}(t,s)$, for every $s\leq t$, which completes the prove of the statement.
	\par Finally, similar to the proof of \eqref{eq-fundamental-relation-betweeen-type1-type2-parabolic-1}, we have that
	\begin{equation}\label{eq-fundamental-relation-betweeen-type1-type2-parabolic-2}
	\|\widetilde{S}(t,s)\widetilde{\Pi}^u(s)\|_{\mathcal{L}(X^*)}\leq
	De^{\nu |-t|} e^{\alpha(-s+t)}, \ t\leq s,
	\end{equation}
	and the proof of the Theorem is complete.
\end{proof}
%

\begin{remark}
	Note that the if the stable family of $\mathcal{S}$ is $X^s_I$, then the stable sets for $\widetilde{\mathcal{S}}$ are $X^s_{II}(t)=[X^s_I(-t)]^*$, for every $t\in \mathbb{R}$.
	\par Therefore if $\mathcal{S}$ admits a NEDI with a null unstable family, then 
	$\widetilde{\mathcal{S}}$ posses a NEDII with a null unstable family.
\end{remark}

\section{Final Remarks and Conclusions}

\par Instead of standard nonuniform exponential dichotomies, we could consider a $\rho$-NED, where $\rho$ is an increasing real function such that $\rho(0)=0$, and
$\lim_{t\to \pm \infty}\rho(t)=\pm \infty$. 
The main difference between a NED and a \textit{$\rho$-NEDII} is in the inequalities that measure the hyperbolicity:
\begin{equation*}\label{eq-rho-NEDI}
\begin{split}
&\|T(t,s)\Pi^s(s)\|_{\mathcal{L}(X)}\leq Me^{\delta|\rho(t)|} e^{-\alpha(\rho(t)-\rho(s))}, 
\ \ t\geq s\\
&\|T(t,s)\Pi^u(s)\|_{\mathcal{L}(X)}\leq Me^{\nu|\rho(t)|} e^{\beta(\rho(t)-\rho(s))}.
\ \ t< s,
\end{split}
\end{equation*}
\par Almost every result of this work can be extended for the case of a $\rho$-NED. The only exception is the robustness result for NEDII, Theorem \ref{th-robustness-invertible-NED2}. In fact, at this situation one has to use a robustness result for a 
$\rho$-NED, which can be found in \cite{Barreira-Valls-Robustness-noninvertible}.
\par For the applications on asymptotic behavior, Section \ref{sec-app-odes} and Subsection \ref{subssec-comparison-scalar-pdes}, one has to consider the additional condition that $\rho^\prime>0$ and in the space of non-homogeneous functions the derivative of $\rho$ will appear as an additional weight:
\begin{equation*}
C_{\eta,\rho}(\mathbb{J})=\big\{\,
b:\mathbb{J}\rightarrow \mathbb{R}:\,
\sup_{r\in \mathbb{J}} \big\{e^{-\eta|r|}[\rho^\prime(r)]^{-1}|b(r)|\big\}	<+\infty\,
\big\}.
\end{equation*}
Then all the results of Section \ref{sec-app-odes} and Subsection \ref{subssec-comparison-scalar-pdes} can be extend for the case $\rho$-NEDII, by a change of variables.

\par In our work, the results when $\mathcal{T}$ admits uniform exponential dichotomies, the non-homogeneous function $b$ must be bounded.
For this type of analysis we recommend \cite{Carvalho-Langa-Robison-book} and  \cite{Longo-Novo-Obaya}.

\section{Acknowledgements}

\par This work was started during the visit of Alexandre N. Oliveira-Sousa to
the Department of Applied Mathematics of the University of Valladolid.
He acknowledges the warm atmosphere and encouragement during this stay.
\par We acknowledge the financial support from the following institutions: J.A. Langa by FEDER Ministerio de Econom\'{\i}a, Industria y Competitividad grants PGC2018-096540-B-I00, and Proyecto I+D+i Programa Operativo FEDER Andalucia US-$1254251$; 
R. Obaya by FEDER Ministerio de Econom\'{\i}a y Competitividad
grants MTM2015-66330-P and RTI2018-096523-B-I00 and by Universidad de Valladolid under
project PIP-TCESC-2020; and A.N. Oliveira-Sousa by S\~ao Paulo Research Foundation (FAPESP) grants 2017/21729-0 and 2018/10633-4, and CAPES grant PROEX-9430931/D.

\bibliographystyle{abbrv}
\bibliography{references_NEDII}

\end{document}